\documentclass[11pt,reqno]{amsart}

\usepackage[utf8]{inputenc}
\usepackage[english]{babel}
\usepackage{ifthen} 

\usepackage{amsmath} 
\usepackage{amssymb} 
\usepackage{mathrsfs} 
\usepackage{mathtools} 
\usepackage{bbm} 
\usepackage{cancel} 
\usepackage{mathdots} 
\usepackage{xfrac} 

\usepackage{amsthm} 
\usepackage[shortlabels]{enumitem}

\usepackage{lipsum} 
\usepackage{verbatim} 
\usepackage{pifont} 
\usepackage[dvipsnames]{xcolor} 
\usepackage{lmodern} 
\usepackage{marginnote} 

\usepackage{wrapfig} 
\usepackage{caption} 

\usepackage{array} 
\usepackage{multirow} 
\usepackage{multicol} 
\usepackage{colortbl} 
\usepackage{makecell} 
\usepackage{arydshln} 

\usepackage{tikz} 
\usetikzlibrary{automata,positioning} 
\usepackage{tkz-euclide} 
\usepackage[customcolors,shade]{hf-tikz} 

\usepackage{geometry} 
\usepackage{titleps} 
\usepackage{footnote} 

\usepackage{hyperref} 
\usepackage{cleveref} 

\newtheorem{theorem}{Theorem}[section]
\newtheorem{proposition}[theorem]{Proposition}
\newtheorem{lemma}[theorem]{Lemma}
\newtheorem{corollary}[theorem]{Corollary}

\newtheorem{example}[theorem]{Example}

\theoremstyle{definition}
\newtheorem{definition}[theorem]{Definition}

\newcolumntype{L}[1]{>{\raggedright\let\newline\\\arraybackslash\hspace{0pt}}m{#1}}
\newcolumntype{C}[1]{>{\centering\let\newline\\\arraybackslash\hspace{0pt}}m{#1}}
\newcolumntype{R}[1]{>{\raggedleft\let\newline\\\arraybackslash\hspace{0pt}}m{#1}}

\usepackage{float}

\hypersetup{
    colorlinks=true,
    linkcolor=blue,
    filecolor=magenta,
    urlcolor=red,
    citecolor=blue,
}
\usetikzlibrary{decorations.pathmorphing}
\usetikzlibrary{decorations.markings}
\usetikzlibrary{decorations.text}

\usetikzlibrary{fit}
\usetikzlibrary{backgrounds}
\usepackage{subcaption}
\usepackage{adjustbox} 

\definecolor{cfefefe}{RGB}{254,254,254}
\definecolor{c0b0b0b}{RGB}{11,11,11}
\definecolor{cc7c6c3}{RGB}{199,198,195}
\definecolor{cdfdedf}{RGB}{223,222,223}
\definecolor{c0a0a09}{RGB}{10,10,9}
\definecolor{c0a0a0a}{RGB}{10,10,10}
\definecolor{c080808}{RGB}{8,8,8}
\definecolor{cbfbcb6}{RGB}{191,188,182}
\definecolor{cc1beb7}{RGB}{193,190,183}
\definecolor{cb5b4b1}{RGB}{181,180,177}
\definecolor{cd8d8d8}{RGB}{216,216,216}
\definecolor{cb8b7b3}{RGB}{184,183,179}
\definecolor{cc4c2c0}{RGB}{196,194,192}
\definecolor{c1b1b1a}{RGB}{27,27,26}
\definecolor{c9e9991}{RGB}{158,153,145}
\definecolor{cbdbab4}{RGB}{189,186,180}
\definecolor{cb9b7b4}{RGB}{185,183,180}
\definecolor{cb9b8b5}{RGB}{185,184,181}
\definecolor{c1a1a19}{RGB}{26,26,25}
\definecolor{ccfcecb}{RGB}{207,206,203}
\definecolor{cb3aea7}{RGB}{179,174,167}
\definecolor{cc9c7c1}{RGB}{201,199,193}
\definecolor{c6b6a68}{RGB}{107,106,104}
\definecolor{cd3d2cf}{RGB}{211,210,207}
\definecolor{c8b8b8b}{RGB}{139,139,139}
\definecolor{ca6a5a2}{RGB}{166,165,162}
\definecolor{caaa9a6}{RGB}{170,169,166}
\definecolor{c6e6c67}{RGB}{110,108,103}
\definecolor{ca5a4a1}{RGB}{165,164,161}
\definecolor{c9d9c9a}{RGB}{157,156,154}
\definecolor{cb8b7b4}{RGB}{184,183,180}
\definecolor{c1c1c1b}{RGB}{28,28,27}
\definecolor{ca39e96}{RGB}{163,158,150}
\definecolor{ca8a6a0}{RGB}{168,166,160}
\definecolor{cbcbbb8}{RGB}{188,187,184}
\definecolor{c505050}{RGB}{80,80,80}
\definecolor{c9f9b93}{RGB}{159,155,147}
\definecolor{cb4b3b0}{RGB}{180,179,176}
\definecolor{c2c2c2b}{RGB}{44,44,43}
\definecolor{ce2e2e3}{RGB}{226,226,227}
\definecolor{ccbc9c3}{RGB}{203,201,195}

\usetikzlibrary{matrix,fit,shapes.geometric,backgrounds,decorations.pathmorphing,arrows.meta,decorations.pathreplacing}
\usepackage{pdflscape}
\usepackage{forest}
\usepackage{needspace}
\theoremstyle{plain}

\newtheorem{examples}[theorem]{Examples}
\newtheorem{paperdefinition}[theorem]{Definition}
\renewenvironment{definition}{\begin{paperdefinition}}{\end{paperdefinition}}
\newtheorem{rankwidthremark}[theorem]{Remark}
\title[Simple Cuts, Simple Landscapes]
{Simple Cuts, Simple Landscapes:\\
Global Synchronization and Rank-Width-One Matrices}

\author[Wu]{Hongjin Wu\textsuperscript{\normalfont *}}
\thanks{\textsuperscript{\normalfont *}%
Centrum Wiskunde \& Informatica, Science Park 123, 1098 XG Amsterdam,
Netherlands, e-mail: \texttt{hongjin-wu@outlook.com}.}

\date{}

\begin{document}
\raggedbottom
\begin{abstract}
We study how the structure of the coupling matrix $\boldsymbol A$
shapes the energy landscape $E_{\boldsymbol A}(\boldsymbol\theta)
=-\frac12\sum_{i,j}A_{ij}\cos(\theta_i-\theta_j)$ of the Kuramoto model.
We prove that, for an entrywise nonnegative coupling matrix $\boldsymbol A$
with connected support, if $\boldsymbol A$ has rank-width one over
$\mathbb R$, then the landscape is ``simple'', i.e., every local minimum
is synchronized. The rank-width bound is sharp, and the result has many
interesting implications for global synchronization.
The key contribution of this work is to uncover a structural
synchronization mechanism beyond the rank-one regime, which has served
as a starting point for earlier analyses of global synchronization.
\end{abstract}
\maketitle

\section{Introduction}
Synchronization of coupled oscillators on graphs raises a basic
question: how do local interactions produce collective agreement?
Answering this question is fundamental to understanding the phenomenon
of synchronization itself and its role in technological systems, such
as power grids, oscillatory neural networks, and Transformer
dynamics~\cite{DorflerChertkovBullo2013,VassilievaEtAl2011,GeshkovskiEtAl2025};
biological systems, such as circadian clocks and neural
circuits~\cite{YamaguchiEtAl2003,DorflerBullo2014};
and social systems, such as rhythmic applause and pedestrian motion
on footbridges~\cite{NedaEtAl2000,StrogatzEtAl2005}, and so on.

The Kuramoto model, introduced by Yoshiki Kuramoto in
1975~\cite{Kuramoto1975}, serves as a canonical model for studying this
question. It is a system of ordinary differential equations of the form
\[
\frac{d\theta_i}{dt}
=\omega_i+\sum_{j=1}^n A_{ij}\sin(\theta_j-\theta_i),
\qquad i=1,\ldots,n.
\]
Here, the $i$th equation describes the evolution of the phase (angle)
$\theta_i(t)$ of oscillator $i$, placed at vertex $i$ of an undirected
graph $G=(V,E)$ with $V=\{1,\ldots,n\}$. Each oscillator has its own
intrinsic frequency $\omega_i\in\mathbb R$, which determines its rotation
rate in the absence of coupling. Its actual rotation rate is also
influenced by phase differences with its neighbors, with coupling
strengths $A_{ij}>0$ along the edges. These coupling strengths form the
weighted adjacency matrix $\boldsymbol A=(A_{ij})$ of $G$; positive
coupling tends to align neighboring phases.

In this work, we assume a common intrinsic frequency,
$\omega_i=\omega$ for all $i$. Passing to a rotating frame via the change
of variables $\theta_i(t)\mapsto\theta_i(t)-\omega t$, we obtain
\begin{equation}\label{kuramoto}
\frac{d\theta_i}{dt}
=\sum_{j=1}^n A_{ij}\sin(\theta_j-\theta_i),
\qquad i=1,\ldots,n.
\end{equation}
To understand whether all oscillators synchronize
(i.e., $\theta_1=\cdots=\theta_n\pmod{2\pi}$) as $t\to\infty$,
the energy perspective offers a clean view: stack all angles
$\theta_1,\ldots,\theta_n$ into a vector
$\boldsymbol\theta=(\theta_1,\ldots,\theta_n)$, and write
\eqref{kuramoto} as
$\dot{\boldsymbol\theta}=-\nabla E_{\boldsymbol A}(\boldsymbol\theta)$,
where
\begin{equation}\label{E}
E_{\boldsymbol A}(\boldsymbol\theta)
:=-\frac12\sum_{i,j=1}^n A_{ij}\cos(\theta_i-\theta_j).
\end{equation}
This shows that the Kuramoto model \eqref{kuramoto} is a gradient flow:
each trajectory $\boldsymbol\theta(t)$ moves in the direction of steepest
descent of $E_{\boldsymbol A}$, and along every trajectory,
$\frac{d}{dt}E_{\boldsymbol A}(\boldsymbol\theta(t))
=-\|\nabla E_{\boldsymbol A}(\boldsymbol\theta(t))\|^2\leq0$,
indicating that the energy is nonincreasing. Therefore, the geometry of
$E_{\boldsymbol A}$ reveals the asymptotic dynamics of~\eqref{kuramoto}.
For rigorous results on the relationship between the asymptotic behavior
of gradient flows and the geometry of their energy functions, we refer
to~\cite{AbsilKurdyka2006,HarauxJendoubi2015}.

One nice property of $E_{\boldsymbol A}$ is that, when the support graph
is connected, every global minimizer is a fully synchronized state. Indeed, since $A_{ij}\geq0$,
\begin{equation}\label{eq:global-energy-gap}
E_{\boldsymbol A}(\boldsymbol\theta)-E_{\boldsymbol A}(\boldsymbol 0)
=\sum_{1\le i<j\le n}A_{ij}\bigl(1-\cos(\theta_i-\theta_j)\bigr)\geq0.
\end{equation}
Equality holds precisely when the phases agree along every edge, which,
by connectedness, means $\theta_1=\cdots=\theta_n\pmod{2\pi}$.

However, the first-order points of $E_{\boldsymbol A}$ other than its
global minimizers can have a much more complicated structure.
In general, $E_{\boldsymbol A}$ may have nonsynchronized local minima.
Nevertheless, in some cases, the energy landscape is rather simple:
every local minimum is fully synchronized. Therefore, the question now is
to understand which coupling matrices yield such a simple landscape and why.

\Needspace{7\baselineskip}
This problem is known as global synchronization
(GS)~\cite{LingXuBandeira,LuSteinerberger2020,KassabovStrogatzTownsend2021,AbdallaEtAl2026,JainMizgerdSawhney2026,LepsveridzeZhang2026}.
In this work, we ask it in the local-minimum sense%
\footnote{Different notions of global synchronization appear in the
literature; see Section~\ref{subsec:sync-properties} for a discussion
of these notions and their relationships.}:

\par\medskip
\noindent\textbf{Question.}
For which entrywise nonnegative coupling matrices $\boldsymbol A$
with connected support does the associated energy $E_{\boldsymbol A}$
have no spurious local minima (i.e., all local minima satisfy
\mbox{$\theta_1=\cdots=\theta_n\pmod{2\pi}$})?
\par\medskip

Whether all local minima of $E_{\boldsymbol A}$ are synchronized
depends on both the topology of the support graph and the coupling weights.

\smallskip
\noindent\textbf{Topology matters.}
In the unweighted case, $A_{ij}=1$ if vertices $i$ and $j$ are adjacent
and $A_{ij}=0$ otherwise, so the energy is determined entirely by the
graph structure. For some graph structures, the corresponding energy
$E_{\boldsymbol A}(\boldsymbol\theta)$ has nonsynchronized local minima.
For example, on a cycle graph $C_n$ with $n\geq5$,
the equally spaced phases $\theta_k=2\pi(k-1)/n$ form a
nonsynchronized strict local minimum modulo global rotation%
\footnote{A common rotation adds the same constant to all phases
and leaves the energy unchanged. Strictness modulo global rotation
means that every sufficiently nearby configuration not related
by such a rotation has strictly higher energy.}
\cite{WileyStrogatzGirvan2006}.

\smallskip
\noindent\textbf{Weights matter.}
In the weighted case, even when the energy associated with the underlying
unweighted graph has no nonsynchronized local minima, assigning positive
weights to its existing edges can create such minima in the weighted
energy $E_{\boldsymbol A}(\boldsymbol\theta)$. For instance, the unweighted complete
graph $K_5$ is almost globally synchronizing: almost every initial phase
configuration converges to a synchronized configuration \cite[Corollary~4.1]{Taylor2012}.
However, assigning weight $1$ to the edges of a five-cycle and a small
positive weight to its diagonals yields a weighted $K_5$ whose coupling
matrix is close to that of the unweighted $C_5$. Taking the diagonal
weights to be $1/10$ produces a nonsynchronized strict local minimum of
$E_{\boldsymbol A}(\boldsymbol\theta)$ modulo global rotation at the equally spaced phase
configuration; see Section~\ref{subsec:weighted-example},
Example~\ref{ex:weighted-k5}
for the explicit configuration and its verification.

Figure~\ref{fig:intro-weights} compares these three graphs.

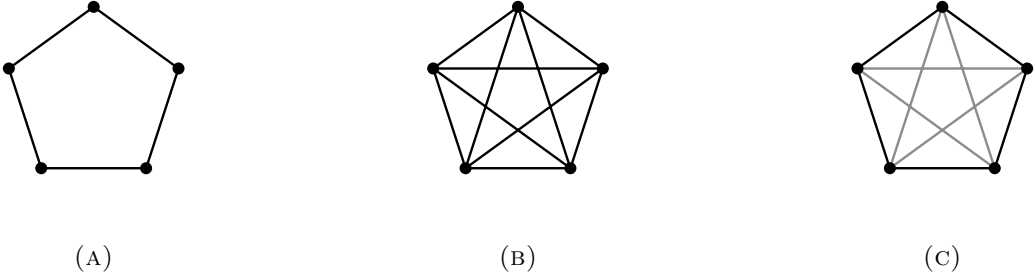
\begin{figure}[!htbp]
\centering
\begingroup
\captionsetup[subfigure]{font=small,justification=centering,skip=4pt}
\tikzset{
  intro edge/.style={draw=black,line width=0.9pt},
  intro weak/.style={draw=black!45,line width=0.9pt},
  intro vertex/.style={circle,fill=black,inner sep=1.6pt}
}
\begin{subfigure}[t]{0.32\linewidth}
\centering
\begin{tikzpicture}[x=1cm,y=1cm]
  \path[use as bounding box] (-2.2,-1.8) rectangle (2.2,1.8);
  \foreach \k in {0,...,4}
    \coordinate (p\k) at ({90-72*\k}:1.18);
  \draw[intro edge] (p0)--(p1)--(p2)--(p3)--(p4)--cycle;
  \foreach \k in {0,...,4} {
    \node[intro vertex] at (p\k) {};
  }
\end{tikzpicture}
\caption{}
\end{subfigure}\hfill
\begin{subfigure}[t]{0.32\linewidth}
\centering
\begin{tikzpicture}[x=1cm,y=1cm]
  \path[use as bounding box] (-2.2,-1.8) rectangle (2.2,1.8);
  \foreach \k in {0,...,4}
    \coordinate (p\k) at ({90-72*\k}:1.18);
  \draw[intro edge] (p0)--(p2)--(p4)--(p1)--(p3)--cycle;
  \draw[intro edge] (p0)--(p1)--(p2)--(p3)--(p4)--cycle;
  \foreach \k in {0,...,4} {
    \node[intro vertex] at (p\k) {};
  }
\end{tikzpicture}
\caption{}
\end{subfigure}\hfill
\begin{subfigure}[t]{0.32\linewidth}
\centering
\begin{tikzpicture}[x=1cm,y=1cm]
  \path[use as bounding box] (-2.2,-1.8) rectangle (2.2,1.8);
  \foreach \k in {0,...,4}
    \coordinate (p\k) at ({90-72*\k}:1.18);
  \draw[intro weak] (p0)--(p2)--(p4)--(p1)--(p3)--cycle;
  \draw[intro edge] (p0)--(p1)--(p2)--(p3)--(p4)--cycle;
  \foreach \k in {0,...,4} {
    \node[intro vertex] at (p\k) {};
  }
\end{tikzpicture}
\caption{}
\end{subfigure}

\caption{Graph structure and coupling weights both affect the energy
landscape. Black edges have weight $1$; gray edges have weight $1/10$.
In (A), the energy of the unweighted cycle $C_5$ has a nonsynchronized
strict local minimum that is also a non-flat first- and second-order point.
In (B), all local minima of the energy of the
unweighted complete graph $K_5$ are synchronized. In (C), the energy
of the weighted $K_5$ has a nonsynchronized strict local minimum.
Strictness is understood modulo global rotation.
See Proposition~\ref{prop:cycles} for (A), Theorem~\ref{main} for (B),
and Example~\ref{ex:weighted-k5} for (C).}
\label{fig:intro-weights}
\endgroup
\end{figure}

\Needspace{7\baselineskip}
\subsection{Main result and its scope}
\label{subsec:contributions}
Our main result shows that if the coupling matrix $\boldsymbol A$ has
rank-width one over $\mathbb R$, then every local minimum of
$E_{\boldsymbol A}$ is a fully synchronized state.

\begin{theorem}\label{main}
Let $\boldsymbol A\in\mathbb R^{n\times n}$ be symmetric, entrywise
nonnegative with zero diagonal, and with
connected support graph. Assume $\boldsymbol{A}$ has rank-width one over $\mathbb R$.
Every local minimum of $E_{\boldsymbol A}(\boldsymbol\theta)$ is a fully synchronized state.
The same conclusion also holds for every non-flat first- and second-order point.
\end{theorem}

Here, a first- and second-order point $\boldsymbol{\theta}$ is a stationary point with
positive semidefinite Hessian of the energy function $E_{\boldsymbol{A}}(\boldsymbol{\theta})$. We call a configuration $\boldsymbol{\theta}$ \emph{non-flat}
if the only straight-line phase variations along which the energy is
locally constant are common rotations; see
Section~\ref{subsec:kuramoto} for the precise definition.

A symmetric entrywise nonnegative matrix $\boldsymbol A$ with zero
diagonal represents a weighted graph: vertices $i$ and $j$ are adjacent
if $A_{ij}>0$, and the edge $ij$ has weight $A_{ij}$.
Informally, $\boldsymbol A$ has rank-width at most one if there exists
a rank-decomposition tree such that the submatrix across each cut
induced by an edge of the tree has rank at most one.
We denote the rank-width over $\mathbb R$ by
$\operatorname{rw}_{\mathbb R}(\boldsymbol A)$.
This condition concerns the ranks of cut submatrices, rather than
the rank of $\boldsymbol A$ itself.
Our proof replaces a global rank-one factorization by rank-one
factorizations across cuts of a decomposition tree, along which
synchronized blocks can be merged.
We use the matrix rank-width framework of Oum~\cite{Oum2012}; see
Section~\ref{subsec:rank-width} for the precise definitions of a
rank-decomposition tree and rank-width.

When all entries of $\boldsymbol A$ belong to $\{0,1\}$,
it is the adjacency matrix of an unweighted simple graph.
In this setting, a graph with adjacency matrix that has rank-width at most one is exactly
distance-hereditary.
These form a well-studied class in structural and algorithmic graph
theory, with close connections to split decomposition and
rank-width~\cite{GioanPaul2012,Oum2005}.
Their definition is based on the preservation of distances:
a graph is \emph{distance-hereditary} if, in every connected induced
subgraph, the distance between any two vertices equals their distance
in the original graph, where distance is the length of a shortest path.
Oum~\cite{Oum2005} established the following characterization,
with cut-ranks computed over the two-element field $\mathbb F_2$.

\begin{theorem}[{\cite[Proposition~7.3]{Oum2005}}]
\label{thm:oum-distance-hereditary}
Let $G$ be a simple graph with adjacency matrix $\boldsymbol A_G$.
Then $G$ is distance-hereditary if and only if
\[
\operatorname{rw}_{\mathbb F_2}(\boldsymbol A_G)\leq1,
\]
where $\operatorname{rw}_{\mathbb F_2}$ denotes rank-width
over the two-element field $\mathbb F_2$.
\end{theorem}

Combining Theorems~\ref{main} and~\ref{thm:oum-distance-hereditary}
gives the following consequence.

\begin{corollary}\label{cor:distance-hereditary}
Let $\boldsymbol A_G$ be the adjacency matrix of a connected
distance-hereditary graph $G$. Every local minimum and every non-flat
first- and second-order point of $E_{\boldsymbol A_G}(\boldsymbol\theta)$ is synchronized.
\end{corollary}

\begin{proof}
Let $G$ be a connected graph with adjacency matrix $\boldsymbol A_G$.
The single-vertex case is immediate, so assume that $G$ has at least
two vertices.
Theorem~\ref{main} gives the conclusion whenever
$\operatorname{rw}_{\mathbb R}(\boldsymbol A_G)=1$.
For any $0/1$ matrix $\boldsymbol M$,
$\operatorname{rank}_{\mathbb R}(\boldsymbol M)\leq1$ if and only if
$\operatorname{rank}_{\mathbb F_2}(\boldsymbol M)\leq1$.
Hence
\[
\operatorname{rw}_{\mathbb R}(\boldsymbol A_G)\leq1
\quad\Longleftrightarrow\quad
\operatorname{rw}_{\mathbb F_2}(\boldsymbol A_G)\leq1.
\]
Since $G$ is connected and has at least two vertices,
both rank-widths are nonzero.
Thus the condition $\operatorname{rw}_{\mathbb R}(\boldsymbol A_G)=1$
is equivalent, by Theorem~\ref{thm:oum-distance-hereditary},
to $G$ being distance-hereditary.
\end{proof}

In the unweighted setting, the rank-width bound underlying
Corollary~\ref{cor:distance-hereditary} is sharp.
Every unweighted cycle $C_n$, $n\geq5$, has real rank-width two and
admits a nonsynchronized configuration that is both a local minimum
and a non-flat first- and second-order point
(Proposition~\ref{prop:cycles}). In the weighted setting,
Example~\ref{ex:weighted-k5} establishes sharpness of both conclusions
of Theorem~\ref{main}: a positively weighted $K_5$ has a coupling matrix
of real rank-width two and admits a nonsynchronized configuration
with both properties.

Cornelsen and Di Stefano~\cite{CornelsenDiStefano2009} proved that the
underlying undirected graph of the transitive closure of any directed
tree is distance-hereditary. Combining this fact with
Corollary~\ref{cor:distance-hereditary} yields the following consequence.

\begin{corollary}\label{cor:tree-transitive-closure}
Let $G$ be the underlying undirected graph of the transitive closure
of a directed tree, and let $\boldsymbol A_G$ be its adjacency matrix.
Every local minimum of $E_{\boldsymbol A_G}(\boldsymbol\theta)$ is a fully synchronized state.
The same conclusion also holds for every non-flat first- and second-order point.
\end{corollary}

\begin{proof}
The graph $G$ is connected and distance-hereditary, so
Corollary~\ref{cor:distance-hereditary} applies.
\end{proof}

\subsection{Synchronization properties and terminology}
\label{subsec:sync-properties}
Different notions of global synchronization appear in the literature.
In this section, we provide an overview of these notions and their relationships.

\begin{paperdefinition}[Almost-global synchronization, AGS]
\label{def:ags}
A connected positively weighted graph $G=(V,E)$ is \emph{almost globally
synchronizing} if, for almost every initial phase configuration in
$(\mathbb R/2\pi\mathbb Z)^V$ with respect to Lebesgue measure, the solution
$\boldsymbol\theta(t)$ of \eqref{kuramoto} converges to a fully synchronized
state. That is,
\[
\lim_{t\to\infty}\boldsymbol\theta(t)=\boldsymbol\theta^*,
\qquad
\theta_i^*\equiv\theta_j^*\pmod{2\pi}
\quad\text{for all }i,j=1,\ldots,n.
\]
\end{paperdefinition}

This dynamical property is called \emph{almost global synchronization}
in \cite{CanaleMonzon2007,CanaleMonzonRobledo2009}. It is called
\emph{global synchronization} (GS) in
\cite[Definition~1.1]{WuBrandesSkeletons},
\cite[Definition~1.1]{AbdallaEtAl2026}, and
\cite[Introduction]{JainMizgerdSawhney2026}.
We distinguish AGS from the following energy-landscape properties:
\begin{enumerate}[label=(\roman*),leftmargin=*,itemsep=2pt]
\item \textbf{Local-minimum GS:} Every local minimum of $E_{\boldsymbol{A}}(\boldsymbol{\theta})$ is a fully synchronized state.
\item \textbf{Second-order GS:} Every first- and second-order point of $E_{\boldsymbol{A}}(\boldsymbol{\theta})$ is a fully synchronized state.
\item \textbf{Non-flat second-order GS:} Every non-flat first- and second-order point of $E_{\boldsymbol{A}}(\boldsymbol{\theta})$ is a fully synchronized state.
\end{enumerate}
\begin{figure}[H]
\centering
\begin{tikzpicture}[
  property/.style={draw,rounded corners=3pt,align=center,
    font=\small\bfseries,inner sep=7pt},
  implication/.style={-{Stealth[length=2mm]},semithick},
  unresolved/.style={implication,dashed}]
\node[property] (second) at (0,3.0) {(ii) Second-order GS};
\node[property] (local) at (-3.4,1.2) {(i) Local-minimum GS};
\node[property] (nonflat) at (3.4,1.2) {(iii) Non-flat\\second-order GS};
\node[property] (ags) at (0,0) {AGS};
\draw[implication] ([xshift=21pt]second.south west) -- ([xshift=9pt]local.north);
\draw[implication] ([xshift=-9pt]local.north) --
  coordinate[midway] (slashleft) ([xshift=3pt]second.south west);
\draw[implication] ([xshift=-21pt]second.south east) -- ([xshift=-9pt]nonflat.north);
\draw[implication] ([xshift=9pt]nonflat.north) --
  coordinate[midway] (slashright) ([xshift=-3pt]second.south east);
\draw[implication] ([xshift=-5pt]second.south) -- ([xshift=-5pt]ags.north);
\draw[implication] ([xshift=5pt]ags.north) --
  coordinate[midway] (slashcenter) ([xshift=5pt]second.south);
\draw[semithick,arrows=-] ([xshift=-2pt,yshift=4pt]slashleft) --
  ([xshift=2pt,yshift=-4pt]slashleft);
\draw[semithick,arrows=-] ([xshift=-2pt,yshift=-4pt]slashright) --
  ([xshift=2pt,yshift=4pt]slashright);
\draw[semithick,arrows=-] ([xshift=-4pt,yshift=3pt]slashcenter) --
  ([xshift=4pt,yshift=-3pt]slashcenter);
\draw[unresolved] (local) -- (ags);
\draw[unresolved] (nonflat) -- (ags);
\node[font=\small] at (-0.72,0.55) {?};
\node[font=\small] at (0.72,0.55) {?};
\end{tikzpicture}
\caption{Relations between synchronization properties for connected
positively weighted graphs. AGS denotes almost-global synchronization.
Solid arrows without a slash denote known
implications; slashed arrows denote implications that do not hold, with
$C_4$ providing a counterexample to all three
\cite[Section~2]{CanaleMonzonRobledo2009}.
Dashed arrows with question
marks indicate implications not established here.}
\label{fig:sync-properties}
\end{figure}
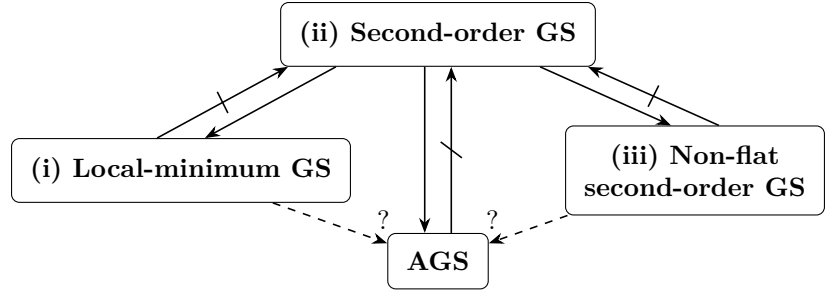

Property (ii) implies (i), since every local minimum satisfies the
first- and second-order conditions, and it implies (iii) directly.
It also implies AGS: real analyticity gives convergence of trajectories
to equilibria, and every nonsynchronized equilibrium then has a negative
Hessian eigenvalue. A center-stable manifold argument shows that the set
of initial conditions converging to these equilibria has measure zero;
see \cite[Introduction]{WuBrandesSkeletons} and the references therein.
Local minima may have flat directions, whereas non-flat first- and
second-order points need not be local minima; these observations alone
do not give an implication between (i) and (iii).

Lepsveridze and Zhang~\cite[Remark~1]{LepsveridzeZhang2026} adopt
property (i) as their definition of global synchronization in their
work on the density conjecture, explicitly distinguishing it from
AGS.
Other results establishing property (i) include
\cite{LingXuBandeira,LuSteinerberger2020}.
Property (ii) follows from the minimum-degree criteria
\cite{Taylor2012,LingXuBandeira,LuSteinerberger2020,KassabovStrogatzTownsend2021},
spectral and random-graph criteria
\cite{KassabovStrogatzTownsend2022,AbdallaEtAl2026,JainMizgerdSawhney2026,AbdallaBandeiraInvernizzi2024,FernandezZhu2026},
weighted landscape criteria
\cite{BandeiraBoumalVoroninski2016,RakotoEndorWaldspurger2025,McRae2026},
and the results on threshold graphs and quasi-threshold graphs \cite{WuBrandesThreshold,WuBrandesSkeletons}.

For the classes considered here, Theorem~\ref{main} and
Corollary~\ref{cor:distance-hereditary} establish (i) and (iii), but not
(ii). We do not establish AGS from either (i) or (iii) alone.

\subsection{Related work}
\label{subsec:related-work}
For unweighted graphs, a central question is the critical minimum-degree
threshold $\mu_c(n-1)$ for local-minimum GS.
Earlier upper bounds on $\mu_c$~\cite{Taylor2012,LingXuBandeira,LuSteinerberger2020}
were improved to $3/4$~\cite{KassabovStrogatzTownsend2021}.
Lepsveridze and Zhang~\cite{LepsveridzeZhang2026} recently lowered this bound
strictly below $3/4$, refuting the sharpness conjecture
in~\cite[Conjecture~5]{BandeiraEtAl2025Randomstrasse}.
Lower bounds from counterexample
constructions~\cite{TownsendStillmanStrogatz2020,YonedaTatsukawaTeramae2021}
were improved to $11/16$~\cite{Canale2022} and most recently to
$0.6929077063\ldots$~\cite{OKeeffe2026}.
The exact value of the local-minimum GS threshold remains open.

Expansion criteria offer another approach.
Abdalla et al.~\cite{AbdallaEtAl2026} proved that a $d$-regular
graph with adjacency matrix $\boldsymbol A$ is almost globally
synchronizing whenever
\[
\left\|\boldsymbol A-\frac{d}{n}
\boldsymbol 1\boldsymbol 1^\top\right\|_{\mathrm{op}}
\leq \alpha d,
\qquad \alpha\leq0.0816.
\]

Rakoto Endor and Waldspurger~\cite{RakotoEndorWaldspurger2025}
and McRae~\cite{McRae2026} established synchronization criteria
based on Laplacian condition numbers, involving the ordinary
Laplacian $\boldsymbol L$ and the normalized Laplacian $\mathcal L$.
For a connected graph, second-order GS holds if
\[
\frac{\lambda_n(\boldsymbol L)}{\lambda_2(\boldsymbol L)}<2
\qquad\text{or}\qquad
\frac{\lambda_n(\mathcal L)}{\lambda_2(\mathcal L)}<2,
\]
where eigenvalues are ordered increasingly.


A common starting point for many of the preceding results is the complete
graph, which satisfies both second-order GS and local-minimum GS.
Since diagonal entries only add a constant to $E_{\boldsymbol A}$,
its adjacency matrix $\boldsymbol 1\boldsymbol 1^\top-\boldsymbol I$
can be replaced by the rank-one matrix $\boldsymbol 1\boldsymbol 1^\top$
without changing the landscape geometry.
Minimum-degree and spectral conditions can then be interpreted as
controlling departures from this rank-one benchmark~\cite[Section~1.1]{WuBrandesSkeletons}.

This perspective leaves a broader structural question: which graphs
ensure local-minimum or second-order GS beyond the rank-one benchmark
and the regimes covered by these criteria?
Recent progress on this question has identified structural mechanisms
for synchronization that do not require density, expansion, or
spectral assumptions.
Wu and Brandes~\cite{WuBrandesThreshold} proved that every connected
threshold graph is second-order globally synchronizing by exploiting
geometric relations between vertices introduced along the graph's
construction sequence.
For the larger class of connected quasi-threshold graphs,
they established second-order GS through a different structural
argument based on the underlying rooted-tree
skeleton~\cite{WuBrandesSkeletons}.
This argument propagates synchronization from the leaves toward
the root, without relying on a construction sequence.

Our work differs from~\cite{WuBrandesThreshold,WuBrandesSkeletons}
in two respects. First, our proof reveals a mechanism of GS distinct
from those based on construction sequences or graph skeletons:
simple cuts induce a simple energy landscape. Specifically, if a
connected positively weighted graph admits a tree-like decomposition
in which the coupling submatrix across each induced cut has rank one,
then the associated energy has no spurious local minima.
Second, the results in~\cite{WuBrandesThreshold,WuBrandesSkeletons}
establish second-order GS, whereas our main conclusion concerns
local-minimum GS. The stronger second-order GS requirement does not hold
for all distance-hereditary graphs.
For example, the four-cycle $C_4$ is distance-hereditary,
and the configuration $(0,\pi/2,\pi,3\pi/2)$, with vertices labeled
cyclically, has zero gradient and zero Hessian. It is therefore a
nonsynchronized first- and second-order point.
Nevertheless, Corollary~\ref{cor:distance-hereditary} shows that every local
minimum is a fully synchronized state for every connected
distance-hereditary graph.
In addition, our result excludes nonsynchronized non-flat first- and
second-order points.

Existing results also address global synchronization on weighted
graphs, including \cite[Theorem~3.4]{McRae2026} and the weighted
extension of \cite[Theorem~1.4]{AbdallaEtAl2026}.
These results do not cover the full class considered here:
real rank-width one constrains the ranks of cut submatrices
but does not imply the required spectral bounds.
Theorem~\ref{main} broadens the rank-one benchmark for local-minimum GS
to real rank-width one.


\subsection{Organization}
\label{subsec:organization}
Section~\ref{sec:setup} introduces the first- and second-order
conditions for the Kuramoto energy function~\eqref{E}.
In particular, we define \textit{non-flatness}, a key concept
in Theorem~\ref{main}.
We then introduce matrix rank-width over $\mathbb R$,
specializing Oum's definition~\cite{Oum2012} to the real field.
Finally, we derive geometric consequences of the stationary
conditions across rank-one cuts.
Section~\ref{sec:proof} proves Theorem~\ref{main}.
The key step is an induction lemma that allows synchronized
vertex sets to be merged.
We then apply this lemma recursively along a rank-decomposition
tree to complete the proof.
Section~\ref{sec:examples} establishes sharpness in the
unweighted and weighted settings using cycles and a positively
weighted complete graph, respectively.

\Needspace{8\baselineskip}
\section{Preliminaries}
\label{sec:setup}
\subsection{Notation}
\label{subsec:notation}
\leavevmode\par
\par\medskip
\noindent\textbf{Sets and vectors.}\par\nobreak
\smallskip
\noindent
For a finite set $W$, let $\mathbb R^W$ denote the space of
real vectors indexed by $W$. For $S\subseteq V$, write
$S^c:=V\setminus S$.
For $\boldsymbol p\in\mathbb R^W$ and $U\subseteq W$, define
\[
\operatorname{supp}(\boldsymbol p):=\{i\in W:p_i\neq0\},
\qquad
\boldsymbol p_U:=(p_i)_{i\in U},
\qquad
p(U):=\sum_{i\in U}p_i.
\]
Let $\boldsymbol 1_U\in\{0,1\}^W$ be the indicator vector
of $U$, with ambient set understood from context.
We write $\boldsymbol 1$ and $\boldsymbol 0$ for the all-ones
and zero vectors of the appropriate dimension, and
$\boldsymbol 0_W$ for the zero vector indexed by $W$.
We use $\langle\cdot,\cdot\rangle$ and $|\cdot|$ for the
Euclidean inner product and norm, respectively, and write
$\boldsymbol 1^\perp
=\{\boldsymbol x:\langle\boldsymbol x,\boldsymbol 1\rangle=0\}$.

\par\medskip
\noindent\textbf{Matrices and graphs.}\par\nobreak
\smallskip
\noindent
For $\boldsymbol A\in\mathbb R^{V\times V}$ and
$U,W\subseteq V$, let $\boldsymbol A[U,W]$ denote the
submatrix with row set $U$ and column set $W$.
We write $\mathcal S_0^n$ for the space of real symmetric
$n\times n$ matrices with zero diagonal.
Throughout this work, the coupling matrix $\boldsymbol A$
belongs to $\mathcal S_0^n$.
Its support graph $G_{\boldsymbol A}$ has vertex set $V$,
with distinct vertices $i,j$ adjacent if and only if $A_{ij}\neq0$.
For a real symmetric matrix $\boldsymbol M$, we write
$\boldsymbol M\succeq0$ if $\boldsymbol M$ is positive semidefinite.
Matrix ranks are taken over $\mathbb R$ unless another
field is specified.

\par\medskip
\noindent\textbf{Phase vectors.}\par\nobreak
\smallskip
\noindent
We write $\mathbb S^1$ for the unit circle in $\mathbb R^2$
and associate with each phase $\theta_i\in\mathbb R$ the vector
$\boldsymbol v_i=(\cos\theta_i,\sin\theta_i)\in\mathbb S^1$.

\subsection{First- and second-order conditions and flat directions}
\label{subsec:kuramoto}
We recall the first- and second-order conditions for the energy
\eqref{E}. Throughout this subsection, we assume only that
$\boldsymbol A\in\mathcal S_0^n$; any additional nonnegativity
assumptions will be stated explicitly.

\noindent
\textbf{First-order condition.}
A point $\boldsymbol\theta$ is called a \emph{first-order point} of $E_{\boldsymbol A}(\boldsymbol\theta)$ if
\begin{equation*}\tag{FOC}\label{eq:foc}
\nabla E_{\boldsymbol A}(\boldsymbol\theta)=\boldsymbol 0,
\qquad \text{i.e.}\qquad
\sum_{j=1}^n A_{ij}\sin(\theta_i-\theta_j)=0,
\quad i=1,\ldots,n.
\end{equation*}
\noindent
\textbf{Second-order condition.}
A point $\boldsymbol\theta$ is called a \emph{second-order point} of $E_{\boldsymbol A}(\boldsymbol\theta)$ if its Hessian
$\boldsymbol H=\nabla^2E_{\boldsymbol A}(\boldsymbol\theta)$ satisfies
\begin{equation*}\tag{SOC}\label{eq:soc}
\boldsymbol H\succeq0,
\end{equation*}
where
\begin{equation}\label{hessian}
H_{ij}
=
\begin{cases}
\displaystyle
\sum_{k=1}^n
A_{ik}\cos(\theta_i-\theta_k),
& i=j,\\[1ex]
-A_{ij}\cos(\theta_i-\theta_j),
& i\neq j.
\end{cases}
\end{equation}

A point satisfying both conditions \eqref{eq:foc} and \eqref{eq:soc} is called a \emph{first- and
second-order point} of $E_{\boldsymbol A}(\boldsymbol\theta)$.
Clearly, every local minimum of $E_{\boldsymbol A}(\boldsymbol\theta)$ satisfies both conditions, as every local minimum is a first- and second-order point.

\begin{definition}
Let $\boldsymbol\theta\in\mathbb R^n$. A nonzero vector
$\boldsymbol d\in\mathbb R^n$ is a \emph{flat direction} of
$E_{\boldsymbol A}(\boldsymbol\theta)$ if
\[
E_{\boldsymbol A}(\boldsymbol\theta+t\boldsymbol d)
=E_{\boldsymbol A}(\boldsymbol\theta)
\qquad\text{for all sufficiently small }t\in\mathbb R.
\]
\end{definition}

We make the following observations.
\begin{enumerate}[label=(\arabic*),leftmargin=*,itemsep=6pt]
\item Every nonzero multiple $\alpha\boldsymbol 1$ is a flat direction,
corresponding to a common rotation of all phases.
We call such directions \emph{trivial}; all other flat directions
are called \emph{nontrivial}.
A point is called \emph{non-flat} if it has no nontrivial flat direction.

\item If $\boldsymbol\theta$ is a first- and second-order point of
$E_{\boldsymbol A}(\boldsymbol\theta)$ and $\boldsymbol d$ is a flat
direction at $\boldsymbol\theta$, then
\[
\nabla^2E_{\boldsymbol A}(\boldsymbol\theta)\boldsymbol d=\boldsymbol 0,
\qquad\text{equivalently,}\qquad
\boldsymbol d\in\ker\nabla^2E_{\boldsymbol A}(\boldsymbol\theta).
\]
\end{enumerate}

\begin{proof}[Proof of (2)]
Write $\boldsymbol H=\nabla^2E_{\boldsymbol A}(\boldsymbol\theta)$.
For a fixed flat direction $\boldsymbol d$, Taylor expansion at
$\boldsymbol\theta$ gives
\[
\begin{aligned}
0&=E_{\boldsymbol A}(\boldsymbol\theta+t\boldsymbol d)
  -E_{\boldsymbol A}(\boldsymbol\theta)\\
 &=t\,\nabla E_{\boldsymbol A}(\boldsymbol\theta)^\top\boldsymbol d
 +\frac{t^2}{2}\boldsymbol d^\top\boldsymbol H\boldsymbol d+t^3r(t).
\end{aligned}
\]
Dividing by $t^2$ for $t\neq0$, we obtain
\[
0=\frac{1}{t}\,\nabla E_{\boldsymbol A}(\boldsymbol\theta)^\top\boldsymbol d+\frac12\boldsymbol d^\top\boldsymbol H\boldsymbol d+t\,r(t).
\]
The first-order condition makes the first
 term zero.
Since $r(t)$ is bounded near zero, we have $t\,r(t)\to0$ as
$t\to0$. Taking limits in the preceding equality therefore gives
$\boldsymbol d^\top\boldsymbol H\boldsymbol d=0$.
Since $\boldsymbol H\succeq0$, we may write
$\boldsymbol H=\boldsymbol B^\top\boldsymbol B$.
Then $\|\boldsymbol B\boldsymbol d\|^2
=\boldsymbol d^\top\boldsymbol H\boldsymbol d=0$,
so $\boldsymbol B\boldsymbol d=\boldsymbol 0$ and hence
$\boldsymbol H\boldsymbol d=\boldsymbol 0$.\end{proof}

Since $\boldsymbol H\boldsymbol 1=\boldsymbol 0$, observation (2)
gives the following sufficient condition for non-flatness.
\begin{lemma}[A sufficient condition for non-flatness]\label{lem:flat-hessian}
Let $\boldsymbol\theta$ be a first- and second-order point of
$E_{\boldsymbol A}(\boldsymbol\theta)$, and let
$\boldsymbol H=\nabla^2E_{\boldsymbol A}(\boldsymbol\theta)$.
If $\operatorname{rank}\boldsymbol H=n-1$, or equivalently
$\ker\boldsymbol H=\mathbb R\boldsymbol 1$, then
$\boldsymbol\theta$ is non-flat.
\end{lemma}
\begin{proof}
By observation (2), every flat direction belongs to
$\ker\boldsymbol H=\mathbb R\boldsymbol 1$, so no flat direction is nontrivial.
\end{proof}

\subsection{Rank-width}
\label{subsec:rank-width}
We specialize Oum's definition of rank-width for symmetric matrices over
an arbitrary field \cite{Oum2012} to $\mathbb R$.

Throughout this paper, we use subcubic trees
$T=(V(T),E(T))$ whose vertices have degree one or three. The degree-one vertices are called \emph{leaves}, and
their set is denoted by $\mathcal L(T)$; the degree-three vertices are
called \emph{internal vertices}.
An edge incident to a leaf is called a \emph{pendant edge};
all other edges are called \emph{nonpendant edges}.
Each edge $e\in E(T)$ partitions $\mathcal L(T)$ into two nonempty
sets, given by the two components of $T-e$.
We call this partition the \emph{cut induced by $e$}.

\begin{definition}[Rank-width over $\mathbb R$]
Let $\boldsymbol A\in\mathcal S_0^n$, with index set $V=\{1,\ldots,n\}$.
\begin{itemize}[label=\textbullet,leftmargin=*,itemsep=3pt]
\item The \emph{cut-rank} of $S\subseteq V$ is
\[
\rho_{\boldsymbol A}(S)
:=\operatorname{rank}_{\mathbb R}\bigl(\boldsymbol A[S,S^c]\bigr).
\]
Since $\boldsymbol A$ is symmetric,
$\rho_{\boldsymbol A}(S)=\rho_{\boldsymbol A}(S^c)$.
\item For $n\geq2$, a \emph{rank-decomposition} of $\boldsymbol A$ is
 a pair $(T,L)$, where $T$ is a subcubic tree and the map
 \mbox{$L:V\to\mathcal L(T)$} is a bijection.
Through $L$, each edge $e$ induces a cut $(S_e,S_e^c)$ of $V$,
where $S_e$ consists of the labels of the leaves in either component of $T-e$.
\item The \emph{width} of a rank-decomposition $(T,L)$ is
\[
\max_{e\in E(T)}\rho_{\boldsymbol A}(S_e).
\]
\item The \emph{rank-width} of $\boldsymbol A$ over $\mathbb R$ is the
minimum width of a rank-decomposition of $\boldsymbol A$:
\[
\operatorname{rw}_{\mathbb R}(\boldsymbol A)
:=\min_{(T,L)}\max_{e\in E(T)}\rho_{\boldsymbol A}(S_e).
\]
For $n=1$, we set $\operatorname{rw}_{\mathbb R}(\boldsymbol A):=0$.
\end{itemize}
\end{definition}

\begin{rankwidthremark}\normalfont
For an unweighted graph with adjacency matrix $\boldsymbol B$,
$\operatorname{rw}_{\mathbb R}(\boldsymbol B)$ equals the
$\mathbb Q$-rank-width studied by Oum, S{\ae}ther, and Vatshelle
\cite[Section~3]{OumSaetherVatshelle2014}, since a rational matrix has
the same rank over $\mathbb Q$ and $\mathbb R$.
This need not equal the usual graph rank-width, which uses
$\mathbb F_2$, although the conditions of having rank-width at most
one coincide, as explained in the proof of
Corollary~\ref{cor:distance-hereditary}.
\end{rankwidthremark}

\begin{examples}\label{ex:rank-width-examples}
\normalfont\leavevmode
\begin{enumerate}[label=(\arabic*),ref=\thetheorem(\arabic*),leftmargin=*,itemsep=6pt]
\item \textbf{The path $P_4$.}
Let $\boldsymbol A$ be the adjacency matrix of the path $P_4$ with edges
$\{1,2\}$, $\{2,3\}$, and $\{3,4\}$.
Figure~\ref{fig:rank-decomposition-example} shows its support graph,
matrix, and a rank-decomposition.
The nonpendant edge $e$ gives $S_e=\{1,2\}$ and $S_e^c=\{3,4\}$;
the highlighted submatrix has rank one.
Every pendant edge gives a nonzero singleton cut, also of rank one.
Thus this decomposition has width one and
$\operatorname{rw}_{\mathbb R}(\boldsymbol A)=1$.

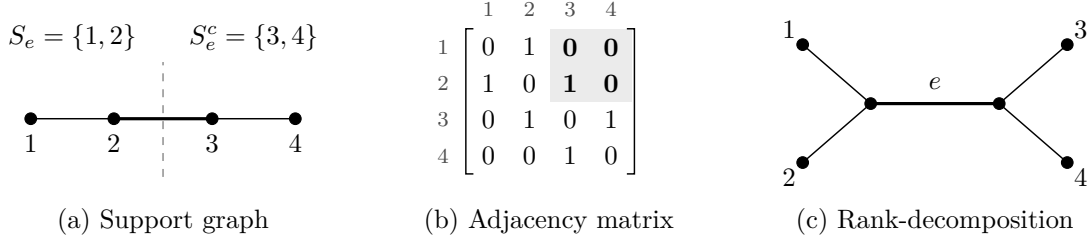
\begin{figure}[H]
\centering
\captionsetup[subfigure]{font=small,labelfont=normalfont,justification=centering,skip=3pt}
\tikzset{
    x=1cm,y=1cm,
    line cap=round,
    every node/.style={font=\small,inner sep=1pt},
    example vertex/.style={circle,fill=black,draw=black,inner sep=1.55pt},
    example edge/.style={draw=black,line width=0.55pt},
    example selected edge/.style={draw=black,line width=1.2pt},
    example index/.style={font=\scriptsize,text=black!70,inner sep=0pt}
}
\begin{minipage}{14.8cm}
\begin{subfigure}[b]{0.31\linewidth}
\centering
\begin{tikzpicture}
  \path[use as bounding box] (0,0) rectangle (4.3,2.7);
  \begin{scope}
    \coordinate (p1) at (0.4,1.05);
    \coordinate (p2) at (1.5,1.05);
    \coordinate (p3) at (2.8,1.05);
    \coordinate (p4) at (3.9,1.05);
    \draw[draw=black!55,dashed,line width=0.45pt]
      (2.15,0.28)--(2.15,1.8);
    \draw[example edge] (p1)--(p2);
    \draw[example selected edge] (p2)--(p3);
    \draw[example edge] (p3)--(p4);
    \foreach \i in {1,2,3,4} {
      \node[example vertex] at (p\i) {};
      \node[below=5pt] at (p\i) {$\i$};
    }
    \node at (0.95,2.13) {$S_e=\{1,2\}$};
    \node at (3.35,2.13) {$S_e^c=\{3,4\}$};
  \end{scope}
\end{tikzpicture}
\caption{Support graph}
\end{subfigure}\hfill
\begin{subfigure}[b]{0.31\linewidth}
\centering
\begin{tikzpicture}
  \path[use as bounding box] (0,0) rectangle (4.3,2.7);
  \begin{scope}[xshift=0.8cm]
    \fill[black!8] (1.36,1.28) rectangle (2.44,2.24);
    \draw[line width=0.55pt]
      (0.35,2.24)--(0.25,2.24)--(0.25,0.32)--(0.35,0.32);
    \draw[line width=0.55pt]
      (2.38,2.24)--(2.48,2.24)--(2.48,0.32)--(2.38,0.32);
    \foreach \row/\col/\entry in {
      1/1/0,1/2/1,1/3/0,1/4/0,
      2/1/1,2/2/0,2/3/1,2/4/0,
      3/1/0,3/2/1,3/3/0,3/4/1,
      4/1/0,4/2/0,4/3/1,4/4/0}
      {
        \ifnum\row<3
          \ifnum\col>2
            \node at ({0.55+0.54*(\col-1)},{2.0-0.48*(\row-1)}) {$\boldsymbol{\entry}$};
          \else
            \node at ({0.55+0.54*(\col-1)},{2.0-0.48*(\row-1)}) {$\entry$};
          \fi
        \else
          \node at ({0.55+0.54*(\col-1)},{2.0-0.48*(\row-1)}) {$\entry$};
        \fi
      }
    \foreach \i in {1,2,3,4} {
      \node[example index] at ({0.55+0.54*(\i-1)},2.5) {$\i$};
      \node[example index] at (-0.05,{2.0-0.48*(\i-1)}) {$\i$};
    }
  \end{scope}
\end{tikzpicture}
\caption{Adjacency matrix}
\end{subfigure}\hfill
\begin{subfigure}[b]{0.31\linewidth}
\centering
\begin{tikzpicture}
  \path[use as bounding box] (0,0) rectangle (4.3,2.7);
  \begin{scope}[xshift=0.15cm]
    \coordinate (tl) at (1.15,1.25);
    \coordinate (tr) at (2.85,1.25);
    \coordinate (leaf1) at (0.25,2.03);
    \coordinate (leaf2) at (0.25,0.47);
    \coordinate (leaf3) at (3.75,2.03);
    \coordinate (leaf4) at (3.75,0.47);
    \draw[example edge] (leaf1)--(tl)--(leaf2);
    \draw[example edge] (leaf3)--(tr)--(leaf4);
    \draw[example selected edge] (tl)--(tr)
      node[midway,above=4pt] {$e$};
    \foreach \name in {tl,tr,leaf1,leaf2,leaf3,leaf4}
      \node[example vertex] at (\name) {};
    \node[above left=2pt] at (leaf1) {$1$};
    \node[below left=2pt] at (leaf2) {$2$};
    \node[above right=2pt] at (leaf3) {$3$};
    \node[below right=2pt] at (leaf4) {$4$};
  \end{scope}
\end{tikzpicture}
\caption{Rank-decomposition}
\end{subfigure}
\end{minipage}
\caption{A rank-decomposition of $P_4$.
The cut $(S_e,S_e^c)$, with $S_e=\{1,2\}$, is represented by the dashed
line in the graph, the bold entries of the shaded submatrix, and the
thick edge $e$ in the tree.}
\label{fig:rank-decomposition-example}
\end{figure}

\item \textbf{The cycle $C_4$.}
Let $\boldsymbol A$ be the adjacency matrix of $C_4$, with vertices labeled cyclically by $1,2,3,4$. Choose a rank-decomposition whose single nonpendant edge induces the cut $(\{1,3\},\{2,4\})$. Its cut matrix $\boldsymbol A[\{1,3\},\{2,4\}]$ is the $2\times2$ all-ones matrix and hence has rank one. Thus this decomposition has width one and $\operatorname{rw}_{\mathbb R}(\boldsymbol A)=1$.

\item \textbf{The cycle $C_n$, $n\geq5$.}\label{ex:cycle-rank-width}
Let $n\geq5$, and let $\boldsymbol A$ be the adjacency matrix
of $C_n$, with vertices labeled cyclically by $1,\ldots,n$.
Then $\operatorname{rw}_{\mathbb R}(\boldsymbol A)=2$. Indeed, every rank-decomposition tree has a nonpendant edge
inducing a cut $(S,S^c)$ with $|S|,|S^c|>1$.
Every such cut of $C_n$ has rank at least two.
Hence every rank-decomposition has width at least two, giving
$\operatorname{rw}_{\mathbb R}(\boldsymbol A)\geq2$.
To show that $\operatorname{rw}_{\mathbb R}(\boldsymbol A)\leq2$, let $S_k=\{1,\ldots,k\}$ for
$2\leq k\leq n-2$. Each cut matrix
$\boldsymbol A[S_k,S_k^c]$ has rank two.
The caterpillar tree in Figure~\ref{fig:cycle-caterpillar}
gives a rank-decomposition whose nonpendant edges induce
exactly these cuts. Thus this decomposition has width two,
giving $\operatorname{rw}_{\mathbb R}(\boldsymbol A)\leq2$. Together, these bounds give
$\operatorname{rw}_{\mathbb R}(\boldsymbol A)=2$.

\end{enumerate}
\end{examples}

\begin{figure}[H]
\centering
\begin{tikzpicture}[
    x=1cm,y=1cm,line cap=round,
    every node/.style={font=\small,inner sep=2pt},
    cat vertex/.style={circle,fill=black,draw=black,inner sep=1.6pt},
    cat edge/.style={draw=black,line width=0.55pt}
]
  \coordinate (a) at (0,0);
  \coordinate (b) at (1.8,0);
  \coordinate (c) at (3.6,0);
  \coordinate (d) at (6.6,0);
  \coordinate (f) at (8.4,0);
  \coordinate (l1) at (-1.15,0.85);
  \coordinate (l2) at (-1.15,-0.85);
  \coordinate (l3) at (1.8,-1.15);
  \coordinate (l4) at (3.6,-1.15);
  \coordinate (ln2) at (6.6,-1.15);
  \coordinate (ln1) at (9.55,0.85);
  \coordinate (ln) at (9.55,-0.85);

  \draw[cat edge] (l1)--(a)--(l2);
  \draw[cat edge] (a)--(b)--(c);
  \draw[cat edge] (c)--(4.65,0);
  \node at (5.1,0) {$\cdots$};
  \draw[cat edge] (5.55,0)--(d)--(f);
  \draw[cat edge] (b)--(l3);
  \draw[cat edge] (c)--(l4);
  \draw[cat edge] (d)--(ln2);
  \draw[cat edge] (ln1)--(f)--(ln);
  \node at (5.1,-1.15) {$\cdots$};

  \foreach \name in {a,b,c,d,f,l1,l2,l3,l4,ln2,ln1,ln}
    \node[cat vertex] at (\name) {};
  \node[above left=2pt] at (l1) {$1$};
  \node[below left=2pt] at (l2) {$2$};
  \node[below=3pt] at (l3) {$3$};
  \node[below=3pt] at (l4) {$4$};
  \node[below=3pt] at (ln2) {$n-2$};
  \node[above right=2pt] at (ln1) {$n-1$};
  \node[below right=2pt] at (ln) {$n$};
  \node[above=4pt] at (0.9,0) {$e_2$};
  \node[above=4pt] at (2.7,0) {$e_3$};
  \node[above=4pt] at (7.5,0) {$e_{n-2}$};
\end{tikzpicture}
\caption{A caterpillar rank-decomposition of $C_n$ (schematic).
The nonpendant edge $e_k$ separates the leaf labels
$S_k=\{1,\ldots,k\}$ from $S_k^c=\{k+1,\ldots,n\}$.
Each such cut has rank two; pendant-edge cuts have rank one.
The dots represent the intervening path vertices and their pendant
leaves; for small $n$, the path is shortened accordingly.}
\label{fig:cycle-caterpillar}
\end{figure}
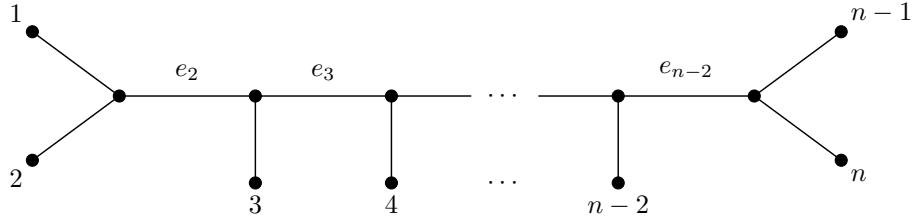

\subsection{First- and second-order conditions on rank-one cuts}
\label{sec:lemmas}
Throughout this subsection, let $\boldsymbol A\in\mathcal S_0^n$.
We derive consequences of the first- and second-order conditions for
cuts, and express them using weighted phasor sums when the cut has rank
one. These results will be used to prove the induction lemma in
Section~\ref{sec:induction-lemma}.
For a cut $(S,S^c)$ with $\rho_{\boldsymbol A}(S)=1$, choose nonzero vectors
$\boldsymbol p\in\mathbb R^S$ and $\boldsymbol q\in\mathbb R^{S^c}$ such that
\begin{equation}\label{eq:rank-one-cut-factorization}
\boldsymbol A[S,S^c]=\boldsymbol p\boldsymbol q^{\mathsf T}.
\end{equation}
Figure~\ref{fig:rank-one-cut} illustrates the edges across such a cut.

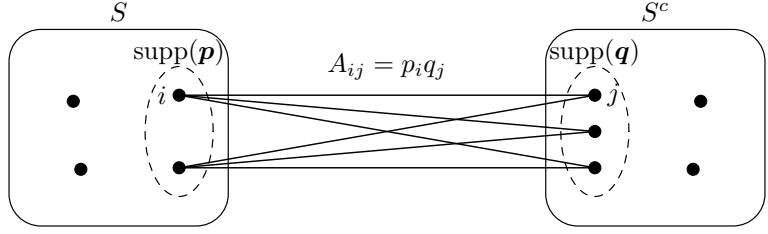
\begin{figure}[!htbp]
\centering
\begin{tikzpicture}[
    x=1cm,y=1cm,line cap=round,line join=round,
    every node/.style={font=\small,inner sep=1.5pt,text=black},
    cut vertex/.style={circle,fill=black,draw=black,inner sep=1.6pt},
    cut set/.style={draw=black,line width=0.45pt},
    cut support/.style={draw=black,dashed,line width=0.45pt},
    crossing edge/.style={draw=black,line width=0.55pt}
]
  \draw[cut set,rounded corners=12pt] (-5,-1.25) rectangle (-2.1,1.35);
  \draw[cut set,rounded corners=12pt] (2.1,-1.25) rectangle (5,1.35);
  \node at (-3.55,1.58) {$S$};
  \node at (3.55,1.58) {$S^c$};
  \draw[cut support] (-2.75,0) ellipse[x radius=0.45,y radius=0.86];
  \draw[cut support] (2.75,0) ellipse[x radius=0.45,y radius=0.86];
  \node at (-2.75,1.05) {$\operatorname{supp}(\boldsymbol p)$};
  \node at (2.75,1.05) {$\operatorname{supp}(\boldsymbol q)$};

  \coordinate (ra1) at (-2.75,0.48);
  \coordinate (ra2) at (-2.75,-0.48);
  \coordinate (rb1) at (2.75,0.48);
  \coordinate (rb2) at (2.75,0);
  \coordinate (rb3) at (2.75,-0.48);
  \foreach \a in {ra1,ra2}
    \foreach \b in {rb1,rb2,rb3}
      \draw[crossing edge] (\a)--(\b);
  \node[above=4pt] at (0,0.48) {$A_{ij}=p_iq_j$};
  \foreach \name in {ra1,ra2,rb1,rb2,rb3}
    \node[cut vertex] at (\name) {};
  \node[left=3pt] at (ra1) {$i$};
  \node[right=3pt] at (rb1) {$j$};

  \node[cut vertex] at (-4.15,0.4) {};
  \node[cut vertex] at (-4.05,-0.5) {};
  \node[cut vertex] at (4.15,0.4) {};
  \node[cut vertex] at (4.05,-0.5) {};
\end{tikzpicture}
\caption{Support of a rank-one cut
$\boldsymbol A[S,S^c]=\boldsymbol p\boldsymbol q^\top$.
The dashed regions contain the vertices with $p_i\neq0$ and $q_j\neq0$,
respectively; the edges across the cut form a complete bipartite graph
between these two sets.
Each crossing edge $ij$ has weight $p_iq_j\neq0$.
Edges within either side are omitted.}
\label{fig:rank-one-cut}
\end{figure}

Summing the first-order condition on any cut gives the
following identity.\par\nopagebreak
\begin{lemma}\label{lemma:first-order-set}
Let $\boldsymbol{A}\in\mathcal S^n_0$. Let $\boldsymbol\theta$ be a first-order point of $E_{\boldsymbol A}(\boldsymbol\theta)$. Then, for every subset $S\subsetneq V$,
\begin{equation}\label{first-order}
\sum_{i\in S}\sum_{j\in S^c} A_{ij}\sin(\theta_i-\theta_j)=0.
\end{equation}
\end{lemma}

\begin{proof}
Summing \eqref{eq:foc} over $i\in S$ gives
\[
0=
\sum_{i\in S}\sum_{j=1}^nA_{ij}\sin(\theta_i-\theta_j)
=
\sum_{i\in S}\sum_{j\in S^c}A_{ij}\sin(\theta_i-\theta_j).
\]
The terms with $i,j\in S$ cancel pairwise because
$A_{ij}=A_{ji}$ and $\sin(\theta_i-\theta_j)=-\sin(\theta_j-\theta_i)$.
\end{proof}

For a cut with rank one, this identity implies that the two weighted
phasor sums are parallel whenever both are nonzero.\par\nopagebreak
\begin{corollary}
\label{lemma:first-order-module2}
Let $\boldsymbol{A}\in\mathcal S^n_0$. Let \(S\subsetneq V\) satisfy \(\rho_{\boldsymbol A}(S)=1\), and choose
\(\boldsymbol p,\boldsymbol q\) as in \eqref{eq:rank-one-cut-factorization}.
Let \(\boldsymbol\theta\) be a first-order point of \(E_{\boldsymbol A}(\boldsymbol\theta)\).
If both vectors
\[
\sum_{i\in S} p_i \boldsymbol v_i
\qquad\text{and}\qquad
\sum_{j\in S^c} q_j \boldsymbol v_j
\]
are nonzero, then they are parallel; that is, there exists \(\mu\neq0\) such that
\[
\sum_{i\in S} p_i \boldsymbol v_i
=
\mu
\sum_{j\in S^c} q_j \boldsymbol v_j.
\]
\end{corollary}

\begin{proof}
Let
\[
\boldsymbol R:=
\begin{pmatrix}
0 & -1\\
1 & 0
\end{pmatrix}
\]
denote the counterclockwise rotation matrix through an angle \(\pi/2\).
Since $\langle\boldsymbol R\boldsymbol v_i,\boldsymbol v_j\rangle
=-\sin(\theta_i-\theta_j)$, the factorization
\eqref{eq:rank-one-cut-factorization} and Lemma~\ref{lemma:first-order-set} give
\[
\left\langle
\boldsymbol R\sum_{i\in S}p_i\boldsymbol v_i,
\sum_{j\in S^c}q_j\boldsymbol v_j
\right\rangle
=-\sum_{i\in S}\sum_{j\in S^c}
A_{ij}\sin(\theta_i-\theta_j)=0.
\]
Since both weighted sums are nonzero vectors in $\mathbb R^2$, this implies that
they are parallel, as claimed.
\end{proof}

The second-order condition gives the following inequality for any
cut.\par\nopagebreak
\begin{lemma}\label{lemma:cut_stability}
Let $\boldsymbol A\in\mathcal S_0^n$, and let $\boldsymbol\theta$ be a
second-order point of $E_{\boldsymbol A}(\boldsymbol\theta)$. Then, for every $S\subsetneq V$,
\[
\sum_{i\in S}\sum_{j\in S^c}
A_{ij}\langle \boldsymbol{v}_i,\boldsymbol{v}_j\rangle \geq 0.
\]
\end{lemma}
\begin{proof}
Set $\boldsymbol x=\boldsymbol 1_S$. We verify that $\boldsymbol x^\top\boldsymbol H\boldsymbol x=\sum_{i\in S}\sum_{j\in S^c}A_{ij}\langle\boldsymbol v_i,\boldsymbol v_j\rangle$. Indeed, by the Hessian formula \eqref{hessian},
\[
\begin{aligned}
\boldsymbol x^\top\boldsymbol H\boldsymbol x
&=\sum_{1\leq i<j\leq n}A_{ij}\cos(\theta_i-\theta_j)(x_i-x_j)^2\\
&=\sum_{i\in S}\sum_{j\in S^c}A_{ij}\cos(\theta_i-\theta_j)\\
&=\sum_{i\in S}\sum_{j\in S^c}A_{ij}\langle\boldsymbol v_i,\boldsymbol v_j\rangle.
\end{aligned}
\]
The second equality holds because $(x_i-x_j)^2=1$ if $i,j$ lie on opposite sides of the cut $(S,S^c)$, and $(x_i-x_j)^2=0$ otherwise.
The conclusion follows from \eqref{eq:soc}.
\end{proof}

For a cut with rank one, this inequality says that the two weighted
phasor sums have a nonnegative inner product.\par\nopagebreak
\begin{corollary}
\label{coro:second-order-module}
Let \(S\subsetneq V\) satisfy \(\rho_{\boldsymbol A}(S)=1\), and choose
\(\boldsymbol p,\boldsymbol q\) as in \eqref{eq:rank-one-cut-factorization}.
Let \(\boldsymbol\theta\) be a second-order point of \(E_{\boldsymbol A}(\boldsymbol\theta)\). Then
\[
\left\langle
\sum_{i\in S}p_i\boldsymbol v_i,
\sum_{j\in S^c}q_j\boldsymbol v_j
\right\rangle
\geq 0.
\]
\end{corollary}

\begin{proof}
By Lemma~\ref{lemma:cut_stability}, the factorization
\eqref{eq:rank-one-cut-factorization}, and bilinearity of the inner product,
\[
0\leq\sum_{i\in S}\sum_{j\in S^c}
A_{ij}\langle\boldsymbol v_i,\boldsymbol v_j\rangle
=\left\langle
\sum_{i\in S}p_i\boldsymbol v_i,
\sum_{j\in S^c}q_j\boldsymbol v_j
\right\rangle.
\]
\end{proof}

\Needspace{9\baselineskip}
For a cut with rank one, combining the parallelism of the two nonzero weighted sums in
Corollary~\ref{lemma:first-order-module2} with their nonnegative inner
product in Corollary~\ref{coro:second-order-module} gives the following
corollary.\par\nopagebreak
\begin{corollary}\label{lemma:1}
Let \(S\subsetneq V\) satisfy \(\rho_{\boldsymbol A}(S)=1\), and choose
\(\boldsymbol p,\boldsymbol q\) as in \eqref{eq:rank-one-cut-factorization}.
Let \(\boldsymbol\theta\) be a first- and second-order point of \(E_{\boldsymbol A}(\boldsymbol\theta)\).
If both vectors
\[
\sum_{i\in S}p_i\boldsymbol v_i
\qquad\text{and}\qquad
\sum_{j\in S^c}q_j\boldsymbol v_j
\]
are nonzero, then they point in the same direction; that is,
\[
\sum_{i\in S}p_i\boldsymbol v_i
=
\mu
\sum_{j\in S^c}q_j\boldsymbol v_j
\qquad\text{for some }\mu>0.
\]
\end{corollary}

\begin{proof}
By Corollary~\ref{lemma:first-order-module2}, there exists $\mu\neq0$ such that
\[
\sum_{i\in S}p_i\boldsymbol v_i
=\mu\sum_{j\in S^c}q_j\boldsymbol v_j.
\]
By Corollary~\ref{coro:second-order-module},
\[
0\leq
\left\langle
\sum_{i\in S}p_i\boldsymbol v_i,
\sum_{j\in S^c}q_j\boldsymbol v_j
\right\rangle
=\mu\left\|\sum_{j\in S^c}q_j\boldsymbol v_j\right\|^2.
\]
Since the second vector is nonzero and $\mu\neq0$, we obtain $\mu>0$.
\end{proof}

\section{Proof of the main theorem}
\label{sec:proof}
\label{subsec:proof-strategy}
\noindent\textbf{Proof strategy.}
Across a cut with rank one, the interaction energy is determined by
two weighted phasor sums, one on each side. The first- and second-order
conditions force these sums to point in the same direction whenever
both are nonzero.

At each internal node of a rank-decomposition tree, the three branches
partition $V$ into $X,Y,Z$. Compatibility of the three rank-one
factorizations allows us to merge synchronized sets $X$ and $Y$:
their common phases must agree. This is the induction lemma proved in
Section~\ref{sec:induction-lemma}, which we apply along the tree in
Section~\ref{sec:main-proof}.

The main difficulty is that a weighted phasor sum may vanish.
When $X$ and $Y$ interact only through $Z$, a vanishing sum on $Z$
would allow $Z$ to rotate without changing the energy.
Non-flatness excludes this directly; local minimality excludes it
through the second-order conditions along this constant-energy path.

\subsection{The induction lemma}
\label{sec:induction-lemma}
We first introduce two basic linear algebra facts that will be used to
prove Lemma~\ref{base-weighted}.

\begin{lemma}[Extension of a rank-one factorization]
\label{lem:rank-one-extension}
Let $\boldsymbol B\in\mathbb R^{m\times r}$ and
$\boldsymbol C\in\mathbb R^{m\times s}$ satisfy $\operatorname{rank}\bigl[\boldsymbol B\;\;\boldsymbol C\bigr]=1$.
Suppose that
$\boldsymbol C=\boldsymbol p\boldsymbol q^\top
\neq\boldsymbol 0$, where $\boldsymbol p\in\mathbb R^m$
and $\boldsymbol q\in\mathbb R^s$.
Then there exists a vector $\boldsymbol f\in\mathbb R^r$
such that
\[
\boldsymbol B=\boldsymbol p\boldsymbol f^\top,
\qquad
\bigl[\boldsymbol B\;\;\boldsymbol C\bigr]
=
\boldsymbol p
\begin{bmatrix}
\boldsymbol f^\top & \boldsymbol q^\top
\end{bmatrix}.
\]
\end{lemma}

\begin{proof}
Choose $j$ with $q_j\neq0$. The $j$th column of $\boldsymbol C$ is $q_j\boldsymbol p\neq\boldsymbol 0$.
Since $\bigl[\boldsymbol B\;\;\boldsymbol C\bigr]$ has rank one,
every column of $\boldsymbol B$ is a scalar multiple of this column,
and hence of $\boldsymbol p$.
Writing the $k$th column as $f_k\boldsymbol p$ gives $\boldsymbol B=\boldsymbol p\boldsymbol f^\top$.
Together with $\boldsymbol C=\boldsymbol p\boldsymbol q^\top$,
this yields the claimed factorization.
\end{proof}

\begin{lemma}[Uniqueness of rank-one factorizations]
\label{lem:rank-one-uniqueness}
Let $\boldsymbol a,\boldsymbol c\in\mathbb R^m$ and
$\boldsymbol b,\boldsymbol d\in\mathbb R^n$ satisfy
$\boldsymbol a\boldsymbol b^\top
=\boldsymbol c\boldsymbol d^\top\neq\boldsymbol 0$.
Then there exists a unique scalar
$\alpha\in\mathbb R\setminus\{0\}$ such that
\[
\boldsymbol a=\alpha\boldsymbol c,
\qquad
\boldsymbol d=\alpha\boldsymbol b.
\]
\end{lemma}

\begin{proof}
Since the common matrix is nonzero, all four vectors are nonzero.
Multiplying the identity on the right by $\boldsymbol b$ gives
\[
\|\boldsymbol b\|^2\boldsymbol a
=(\boldsymbol d^\top\boldsymbol b)\boldsymbol c.
\]
The left-hand side is nonzero, so $\boldsymbol d^\top\boldsymbol b\neq0$.
Thus $\boldsymbol a=\alpha\boldsymbol c$ with $\alpha:=(\boldsymbol d^\top\boldsymbol b)/\|\boldsymbol b\|^2\neq0$.
Substituting this into the assumed identity gives $\boldsymbol c(\alpha\boldsymbol b-\boldsymbol d)^\top =\boldsymbol 0$.
Since $\boldsymbol c\neq\boldsymbol 0$, we obtain $\boldsymbol d=\alpha\boldsymbol b$.
The uniqueness of $\alpha$ follows from $\boldsymbol c\neq\boldsymbol 0$.
\end{proof}

\begin{lemma}[Induction lemma]\label{base-weighted}
Let \(\boldsymbol A\in\mathcal S_0^n\) be entrywise nonnegative, and assume that the support graph of \(\boldsymbol A\) is connected.
Let $\boldsymbol\theta$ be either a local minimum or a non-flat first-
and second-order point of $E_{\boldsymbol A}(\boldsymbol\theta)$.
Let $V=X\cup Y\cup Z$ be a partition into
three nonempty sets satisfying
\[
\rho_{\boldsymbol A}(X)=\rho_{\boldsymbol A}(Y)=\rho_{\boldsymbol A}(Z)=1.
\]
Suppose that \(\boldsymbol\theta\) is synchronized on both \(X\) and \(Y\);
that is, there exist
\(\theta_X,\theta_Y\in\mathbb R\) such that
\[
\theta_i\equiv\theta_X\pmod{2\pi}\quad(i\in X),
\qquad
\theta_j\equiv\theta_Y\pmod{2\pi}\quad(j\in Y).
\]
Then $\theta_X\equiv\theta_Y\pmod{2\pi}$;  equivalently, $\boldsymbol\theta$ is synchronized on $X\cup Y$.
\end{lemma}

\begin{proof}
We show that $\theta_X\equiv\theta_Y\pmod{2\pi}$ by considering
three mutually exclusive and exhaustive cases, addressed by the
three auxiliary lemmas proved below.
If $\boldsymbol A[X,Y]=\boldsymbol 0$, connectedness forces both
$\boldsymbol A[X,Z]$ and $\boldsymbol A[Y,Z]$ to be nonzero;
otherwise, $X$ or $Y$ would be disconnected from its complement.
If
$\boldsymbol A[X,Y]\neq\boldsymbol 0$, then at least one of
$\boldsymbol A[X,Z]$ and $\boldsymbol A[Y,Z]$ must be nonzero;
otherwise, $X\cup Y$ would be disconnected from the nonempty
set $Z$, contradicting the connectedness of the support graph.

If \(\boldsymbol A[X,Y]=\boldsymbol 0\), the conclusion follows from
Lemma~\ref{lem:merge-zero}.
If \(\boldsymbol A[X,Y]\neq\boldsymbol 0\) and both
\(\boldsymbol A[X,Z]\) and \(\boldsymbol A[Y,Z]\) are nonzero, the conclusion
follows from Lemma~\ref{lem:merge-compatible}.
In the remaining case, exactly one of
\(\boldsymbol A[X,Z]\) and \(\boldsymbol A[Y,Z]\) is nonzero, and the conclusion
follows from Lemma~\ref{lem:merge-one-sided}.
\end{proof}

We now prove the three auxiliary lemmas used in the proof of
Lemma~\ref{base-weighted}, one for each interaction pattern.
To simplify their proofs,
we introduce the following notation. Since $\boldsymbol\theta$ is synchronized on $X$ and $Y$, let
$\boldsymbol v_X,\boldsymbol v_Y\in\mathbb S^1$ denote their common unit vectors;
thus $\boldsymbol v_i=\boldsymbol v_X$ for $i\in X$, and
$\boldsymbol v_j=\boldsymbol v_Y$ for $j\in Y$.

By symmetry of $\boldsymbol A$,
$\rho_{\boldsymbol A}(X\cup Y)=\rho_{\boldsymbol A}(Z)=1$.
Hence there exist \(\boldsymbol p\in\mathbb R^{X\cup Y}\) and \(\boldsymbol q\in\mathbb R^Z\) such that
\begin{equation}\label{XYZ}
    \boldsymbol A[X\cup Y,Z]=\boldsymbol p\boldsymbol q^\top.
\end{equation}
Since \(A_{ij}\geq 0\) for all \(i,j=1,\ldots,n\), we may choose the factorization in \eqref{XYZ} so that \(p_i,q_j\geq0\) for all \(i\in X\cup Y\) and \(j\in Z\).

Figure~\ref{fig:interaction-cases} illustrates the three interaction
patterns treated in the auxiliary lemmas.

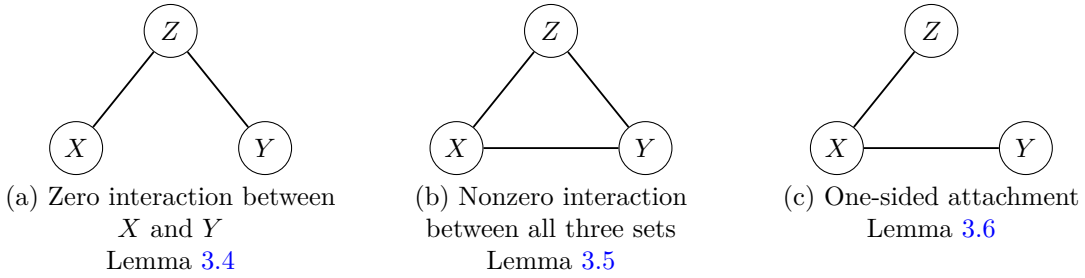
\begin{figure}[H]
\centering
\captionsetup[subfigure]{font=small,labelfont=normalfont,justification=centering,skip=3pt}
\tikzset{
    x=1cm,y=1cm,
    block/.style={circle,draw,fill=white,minimum size=7mm,inner sep=0pt},
    interaction/.style={line width=0.7pt},
    every node/.style={font=\small}
}
\begin{minipage}{14.8cm}
\begin{subfigure}[t]{0.32\linewidth}
\centering
\begin{tikzpicture}
  \path[use as bounding box] (-0.4,-0.35) rectangle (2.9,1.95);
\begin{scope}
    \node[block] (x) at (0,0) {$X$};
    \node[block] (y) at (2.5,0) {$Y$};
    \node[block] (z) at (1.25,1.55) {$Z$};
    \draw[interaction] (x) -- (z) -- (y);
\end{scope}
\end{tikzpicture}
\caption[Zero interaction between $X$ and $Y$]{Zero interaction between $X$ and $Y$\\Lemma~\ref{lem:merge-zero}}
\end{subfigure}\hfill
\begin{subfigure}[t]{0.32\linewidth}
\centering
\begin{tikzpicture}
  \path[use as bounding box] (-0.4,-0.35) rectangle (2.9,1.95);
\begin{scope}
    \node[block] (x) at (0,0) {$X$};
    \node[block] (y) at (2.5,0) {$Y$};
    \node[block] (z) at (1.25,1.55) {$Z$};
    \draw[interaction] (x) -- (y) -- (z) -- (x);
\end{scope}
\end{tikzpicture}
\caption[Nonzero interaction between all three sets]{Nonzero interaction between all three sets\\Lemma~\ref{lem:merge-compatible}}
\end{subfigure}\hfill
\begin{subfigure}[t]{0.32\linewidth}
\centering
\begin{tikzpicture}
  \path[use as bounding box] (-0.4,-0.35) rectangle (2.9,1.95);
\begin{scope}
    \node[block] (x) at (0,0) {$X$};
    \node[block] (y) at (2.5,0) {$Y$};
    \node[block] (z) at (1.25,1.55) {$Z$};
    \draw[interaction] (y) -- (x) -- (z);
\end{scope}
\end{tikzpicture}
\caption[One-sided attachment]{One-sided attachment\\Lemma~\ref{lem:merge-one-sided}}
\end{subfigure}
\end{minipage}
\caption{Each node represents a vertex set.
Two nodes are linked by a line if and only if the support graph
has at least one edge with one endpoint in each of the corresponding vertex sets.
Panel~(c) also covers the case obtained by exchanging $X$ and $Y$.}
\label{fig:interaction-cases}
\end{figure}

\begin{lemma}[Zero interaction between the synchronized blocks $X$ and $Y$]
\label{lem:merge-zero}
Under the assumptions and notation of Lemma~\ref{base-weighted}, suppose that
$\boldsymbol A[X,Y]=\boldsymbol 0$. Then \(\theta_X\equiv\theta_Y\pmod{2\pi}\).
\end{lemma}
\begin{proof}
Since the support graph is connected and
$\boldsymbol A[X,Y]=\boldsymbol 0$, both $\boldsymbol A[X,Z]$ and
$\boldsymbol A[Y,Z]$ are nonzero. By \eqref{XYZ}, the vectors
$\boldsymbol p_X$ and $\boldsymbol p_Y$ are therefore nonzero and
entrywise nonnegative, so $p(X)>0$ and $p(Y)>0$.

\smallskip
\noindent\textbf{Step 1. We show that $\sum_{j\in Z}q_j\boldsymbol v_j\neq0$.}
Suppose, to the contrary, that
$\boldsymbol w_Z:=\sum_{j\in Z}q_j\boldsymbol v_j=\boldsymbol 0$.
For disjoint vertex sets $U,W$, denote their interaction energy by
\[
E_{U\leftrightarrow W}(\boldsymbol\phi)
:=-\sum_{i\in U}\sum_{j\in W}A_{ij}\cos(\phi_i-\phi_j),
\]
and write
\[
E_U(\boldsymbol\phi)
:=-\frac12\sum_{i,j\in U}A_{ij}\cos(\phi_i-\phi_j)
\]
for the energy of interactions within $U$.
Set $\boldsymbol\theta^t=\boldsymbol\theta+t\boldsymbol 1_Z$,
and let $\boldsymbol Q_t$ be the rotation matrix through angle $t$.
Then the weighted sum on $Z$ remains zero:
\[
\boldsymbol w_Z(t)
:=\sum_{j\in Z}q_j
\begin{pmatrix}\cos(\theta_j+t)\\\sin(\theta_j+t)\end{pmatrix}
=\boldsymbol Q_t\boldsymbol w_Z=\boldsymbol 0.
\]
Expanding the energy according to the partition $X,Y,Z$ gives
\[
\begin{aligned}
E_{\boldsymbol A}(\boldsymbol\theta^t)
&=E_X(\boldsymbol\theta^t)+E_Y(\boldsymbol\theta^t)
  +E_Z(\boldsymbol\theta^t)\\
&\quad+E_{X\leftrightarrow Y}(\boldsymbol\theta^t)
  +E_{X\leftrightarrow Z}(\boldsymbol\theta^t)
  +E_{Y\leftrightarrow Z}(\boldsymbol\theta^t).
\end{aligned}
\]
The first three terms are independent of $t$, since all phase differences
within each set are unchanged. Moreover,
$E_{X\leftrightarrow Y}(\boldsymbol\theta^t)=0$ because
$\boldsymbol A[X,Y]=\boldsymbol 0$. By \eqref{XYZ},
\[
\begin{aligned}
E_{X\leftrightarrow Z}(\boldsymbol\theta^t)
&=-\sum_{i\in X}\sum_{j\in Z}p_iq_j\cos(\theta_i-\theta_j-t)\\
&=-\left\langle\sum_{i\in X}p_i\boldsymbol v_i,
                 \boldsymbol w_Z(t)\right\rangle=0.
\end{aligned}
\]
The same calculation gives
$E_{Y\leftrightarrow Z}(\boldsymbol\theta^t)=0$. Thus
\[
E_{\boldsymbol A}(\boldsymbol\theta^t)
=E_{\boldsymbol A}(\boldsymbol\theta)
\qquad\text{for every }t\in\mathbb R.
\]
\smallskip
\noindent\textit{Case (a): $\boldsymbol\theta$ is a non-flat first- and second-order point.}
The energy identity above contradicts non-flatness because $\boldsymbol 1_Z$ is a
nontrivial flat direction.

\smallskip
\noindent\textit{Case (b): $\boldsymbol\theta$ is a local minimum.}
Since $\boldsymbol\theta$ is a local minimum, there is an open
neighborhood $\mathcal U$ of $\boldsymbol\theta$ on which
$E_{\boldsymbol A}(\boldsymbol\phi)\geq E_{\boldsymbol A}(\boldsymbol\theta)$ for every $\boldsymbol\phi\in\mathcal U$.
For sufficiently small $t$, we have $\boldsymbol\theta^t\in\mathcal U$
and $E_{\boldsymbol A}(\boldsymbol\theta^t)=E_{\boldsymbol A}(\boldsymbol\theta)$.
Thus $\boldsymbol\theta^t$ also minimizes the energy on $\mathcal U$,
so it is a local minimum. Consequently,
$\boldsymbol H(t):=\nabla^2E_{\boldsymbol A}(\boldsymbol\theta^t)\succeq0$.
Since $\boldsymbol A[X,Y]=\boldsymbol 0$ and $\boldsymbol w_Z(t)=\boldsymbol 0$,
\[
\begin{aligned}
\boldsymbol 1_X^\top\boldsymbol H(t)\boldsymbol 1_X
&=\sum_{i\in X}\sum_{j\in Z}p_iq_j\cos(\theta_X-\theta_j-t)\\
&=p(X)\langle\boldsymbol v_X,\boldsymbol w_Z(t)\rangle=0.
\end{aligned}
\]
Since $\boldsymbol H(t)\succeq0$, the vanishing quadratic form implies
$\boldsymbol H(t)\boldsymbol 1_X=\boldsymbol 0$
for every sufficiently small $t$.
Since $\boldsymbol A[X,Z]\neq\boldsymbol 0$, we have $p(X)>0$
and can fix $j\in Z$ with $q_j>0$.
The $j$th coordinate of this identity gives
\[
0=(\boldsymbol H(t)\boldsymbol 1_X)_j
=-\sum_{i\in X}p_iq_j\cos(\theta_j+t-\theta_X)
=-p(X)q_j\cos(\theta_j+t-\theta_X).
\]
Since $p(X)q_j>0$, it follows that
\[
\cos(\theta_j+t-\theta_X)=0
\qquad\text{for every sufficiently small }t.
\]
This is impossible because the cosine function does not vanish on
any interval.
Hence $\boldsymbol w_Z\neq\boldsymbol 0$ in both cases.

\smallskip
\noindent\textbf{Step 2. We show that $\sum_{j\in Z}q_j\boldsymbol v_j=\lambda_X\boldsymbol v_X$ for some $\lambda_X>0$.}

Since \(\rho_{\boldsymbol A}(X)=1\), the cut matrix
\(\boldsymbol A[X,Y\cup Z]\) has rank one. By \eqref{XYZ},
we have $$\boldsymbol A[X,Z]=\boldsymbol p_X\boldsymbol q^\top.$$
Together with
\(\boldsymbol A[X,Y]=\boldsymbol 0\), this gives
\begin{equation}\label{X,YZ}
\begin{aligned}
\boldsymbol A[X,Y\cup Z]
&=
\bigl[\boldsymbol A[X,Y]\;\; \boldsymbol A[X,Z]\bigr] \\
&=
\bigl[\boldsymbol 0\;\; \boldsymbol p_X\boldsymbol q^\top\bigr]
=
\boldsymbol p_X
\begin{bmatrix}
\boldsymbol 0_Y^\top & \boldsymbol q^\top
\end{bmatrix}.
\end{aligned}
\end{equation}

By Step~1, $\sum_{j\in Z}q_j\boldsymbol v_j\neq\boldsymbol 0$.
Moreover, since \(\boldsymbol\theta\) is synchronized on \(X\),
\[
\sum_{i\in X}p_i\boldsymbol v_i
=
p(X)\boldsymbol v_X.
\]
Since $p(X)>0$ and $\boldsymbol v_X\neq\boldsymbol 0$,
this weighted sum is nonzero.

By Corollary~\ref{lemma:1}, since both
$\sum_{i\in X}p_i\boldsymbol v_i$ and
$\sum_{j\in Z}q_j\boldsymbol v_j$ are nonzero, the factorization
\eqref{X,YZ} gives a scalar $\mu>0$ such that
\[
p(X)\boldsymbol v_X
=
\mu\sum_{j\in Z}q_j\boldsymbol v_j.
\]
Since $p(X)>0$ and $\mu>0$, we conclude that
\begin{equation}\label{A}
\sum_{j\in Z}q_j\boldsymbol v_j=\lambda_X\boldsymbol v_X,
\qquad \lambda_X:=\frac{p(X)}{\mu}>0.
\end{equation}
\smallskip
\noindent\textbf{Step 3. We show that $\sum_{j\in Z}q_j\boldsymbol v_j=\lambda_Y\boldsymbol v_Y$ for some $\lambda_Y>0$.}
Similarly, switching the roles of $X$ and $Y$ in Step~2, we obtain
\begin{equation}\label{B}
\sum_{j\in Z}q_j\boldsymbol v_j=\lambda_Y\boldsymbol v_Y
\qquad\text{for some }\lambda_Y>0.
\end{equation}
\smallskip
\noindent\textbf{Step 4. We show that $\boldsymbol\theta$ is synchronized on $X\cup Y$.}
From \eqref{A}, \eqref{B}, and
$\|\boldsymbol v_X\|=\|\boldsymbol v_Y\|=1$, we obtain
$\lambda_X=\lambda_Y$ and hence $\boldsymbol v_X=\boldsymbol v_Y$.
Thus $\theta_X\equiv\theta_Y\pmod{2\pi}$.

\end{proof}

\begin{lemma}[Nonzero interaction between all three sets]
\label{lem:merge-compatible}
Under the assumptions and notation of Lemma~\ref{base-weighted}, suppose that
$\boldsymbol A[X,Y]\neq\boldsymbol 0$,
$\boldsymbol A[X,Z]\neq\boldsymbol 0$, and
$\boldsymbol A[Y,Z]\neq\boldsymbol 0$. Then \(\theta_X\equiv\theta_Y\pmod{2\pi}\).
\end{lemma}

\begin{proof}

\noindent\textbf{Step 1. We show that $\boldsymbol A[X,Y]=\boldsymbol p_X\boldsymbol f^\top$ for some $\boldsymbol f\in\mathbb R^Y$.}

By \eqref{XYZ}, we have
$\boldsymbol A[X,Z]=\boldsymbol p_X\boldsymbol q^\top
\neq\boldsymbol 0$.
Since $\rho_{\boldsymbol A}(X)=1$, the submatrix
$\boldsymbol A[X,Y\cup Z]
=\bigl[\boldsymbol A[X,Y]\;\;\boldsymbol A[X,Z]\bigr]$
has rank one.
Applying Lemma~\ref{lem:rank-one-extension}, we obtain
a vector $\boldsymbol f\in\mathbb R^Y$ such that
\begin{equation}\label{ha}
\boldsymbol A[X,Y]=\boldsymbol p_X\boldsymbol f^\top.
\end{equation}
Consequently,
\[
\boldsymbol A[X,Y\cup Z]
=
\boldsymbol p_X
\begin{bmatrix}
\boldsymbol f^\top & \boldsymbol q^\top
\end{bmatrix}.
\]

\noindent\textbf{Step 2. We show that $\boldsymbol A[Y,X]=\boldsymbol p_Y\boldsymbol e^\top$ for some $\boldsymbol e\in\mathbb R^X$.}

Similarly, switching the roles of $X$ and $Y$ in Step~1, there exists
$\boldsymbol e\in\mathbb R^X$ such that
\begin{equation}\label{ha2}
\boldsymbol A[Y,X]=\boldsymbol p_Y\boldsymbol e^\top.
\end{equation}
Consequently,
\[
\boldsymbol A[Y,X\cup Z]
=
\boldsymbol p_Y
\begin{bmatrix}
\boldsymbol e^\top & \boldsymbol q^\top
\end{bmatrix}.
\]

\noindent\textbf{Step 3. We show that $\boldsymbol f=\alpha\boldsymbol p_Y$ and $\boldsymbol e=\alpha\boldsymbol p_X$ for some $\alpha>0$.}

By \eqref{ha} and \eqref{ha2}, and by symmetry of $\boldsymbol A$, we have
$\boldsymbol A[Y,X]=\boldsymbol A[X,Y]^\top$. Thus,
$\boldsymbol p_Y\boldsymbol e^\top=\boldsymbol f\boldsymbol p_X^\top$.
Applying Lemma~\ref{lem:rank-one-uniqueness}, we obtain
a scalar $\alpha\neq 0$ such that
\[
\boldsymbol f=\alpha\boldsymbol p_Y,
\qquad
\boldsymbol e=\alpha\boldsymbol p_X.
\]
Consequently,
\[
\boldsymbol A[X,Y]
=
\boldsymbol p_X\boldsymbol f^\top
=
\alpha\boldsymbol p_X\boldsymbol p_Y^\top.
\]
Since $\boldsymbol p_X$ and $\boldsymbol p_Y$ are
nonzero and entrywise nonnegative, and
$\boldsymbol A[X,Y]$ is entrywise nonnegative,
it follows that $\alpha>0$.
We can therefore write the two cut matrices as
\begin{equation}\label{eq:compatible-cut-factorizations}
\begin{aligned}
\boldsymbol A[X,Y\cup Z]
&=
\boldsymbol p_X
\begin{bmatrix}
\alpha\boldsymbol p_Y^\top & \boldsymbol q^\top
\end{bmatrix},\\
\boldsymbol A[Y,X\cup Z]
&=
\boldsymbol p_Y
\begin{bmatrix}
\alpha\boldsymbol p_X^\top & \boldsymbol q^\top
\end{bmatrix}.
\end{aligned}
\end{equation}

\smallskip
\noindent\textbf{Step 4. We show that $\boldsymbol\theta$ is synchronized on $X\cup Y$.}

For the cut $(X,Y\cup Z)$ with the factorization in
\eqref{eq:compatible-cut-factorizations},
the two weighted sums are
\[
\sum_{i\in X}p_i\boldsymbol v_i
=p(X)\boldsymbol v_X
\qquad\text{and}\qquad
\alpha\sum_{j\in Y}p_j\boldsymbol v_j
+\sum_{k\in Z}q_k\boldsymbol v_k
=\alpha p(Y)\boldsymbol v_Y+\sum_{k\in Z}q_k\boldsymbol v_k.
\]
Here we used $\boldsymbol v_i=\boldsymbol v_X$ for $i\in X$ and
$\boldsymbol v_j=\boldsymbol v_Y$ for $j\in Y$, since
$\boldsymbol\theta$ is synchronized on each of $X$ and $Y$, together
with $p(X)=\sum_{i\in X}p_i$ and $p(Y)=\sum_{j\in Y}p_j$.
By Corollary~\ref{lemma:1}, if both sums are nonzero, they point in
the same direction. The first sum is nonzero because $p(X)>0$ and
$\boldsymbol v_X\neq\boldsymbol 0$.

We next show that the second sum is nonzero. Set
$\boldsymbol w_Z:=\sum_{k\in Z}q_k\boldsymbol v_k$.
Suppose, to the contrary, that
$\alpha p(Y)\boldsymbol v_Y+\boldsymbol w_Z=\boldsymbol 0$, so that
$\boldsymbol w_Z=-\alpha p(Y)\boldsymbol v_Y$.
Applying Corollary~\ref{coro:second-order-module} to the cuts
$(Y,X\cup Z)$ and $(Z,X\cup Y)$, using
\eqref{eq:compatible-cut-factorizations} and \eqref{XYZ}, respectively, gives
\[
\begin{aligned}
0
&\leq\left\langle p(Y)\boldsymbol v_Y,
                 \alpha p(X)\boldsymbol v_X+\boldsymbol w_Z\right\rangle\\
&=\alpha p(Y)\bigl(p(X)\langle\boldsymbol v_X,\boldsymbol v_Y\rangle-p(Y)\bigr),\\
0
&\leq\left\langle\boldsymbol w_Z,
                 p(X)\boldsymbol v_X+p(Y)\boldsymbol v_Y\right\rangle\\
&=-\alpha p(Y)\bigl(p(X)\langle\boldsymbol v_X,\boldsymbol v_Y\rangle+p(Y)\bigr).
\end{aligned}
\]
Adding these inequalities yields $0\leq-2\alpha p(Y)^2<0$, a contradiction.
Thus the second sum is also nonzero.
Applying Corollary~\ref{lemma:1} to these two nonzero sums
and using $p(X)>0$, we obtain a scalar $\lambda_X>0$ such that
\begin{equation}\label{eq:lambdaX}
\lambda_X\boldsymbol v_X
=
\alpha p(Y)\boldsymbol v_Y
+
\sum_{k\in Z}q_k\boldsymbol v_k.
\end{equation}
Similarly, applying the same argument to the cut $(Y,X\cup Z)$
gives a scalar $\lambda_Y>0$ such that
\begin{equation}\label{eq:lambdaY}
\lambda_Y\boldsymbol v_Y
=
\alpha p(X)\boldsymbol v_X
+
\sum_{k\in Z}q_k\boldsymbol v_k.
\end{equation}

Eliminating the common term
$\sum_{k\in Z}q_k\boldsymbol v_k$ from
\eqref{eq:lambdaX} and \eqref{eq:lambdaY}, we obtain
\[
\left(
\lambda_X+\alpha p(X)
\right)\boldsymbol v_X
=
\left(
\lambda_Y+\alpha p(Y)
\right)\boldsymbol v_Y.
\]
Since $\alpha>0$, $p(X)>0$, $p(Y)>0$,
$\lambda_X>0$, and $\lambda_Y>0$, both coefficients are strictly positive.
Therefore, $\boldsymbol v_X=\boldsymbol v_Y$,
and hence $\theta_X\equiv\theta_Y\pmod{2\pi}$.

\end{proof}

\begin{lemma}[One-sided attachment]
\label{lem:merge-one-sided}
Under the assumptions and notation of Lemma~\ref{base-weighted}, suppose that
$\boldsymbol A[X,Y]\neq\boldsymbol 0$ and exactly one of
\(\boldsymbol A[X,Z]\) and \(\boldsymbol A[Y,Z]\) is nonzero.
Then \(\theta_X\equiv\theta_Y\pmod{2\pi}\).
\end{lemma}
\begin{proof}
By symmetry, we may assume that
\[
\boldsymbol A[X,Z]\neq\boldsymbol 0,
\qquad
\boldsymbol A[Y,Z]=\boldsymbol 0.
\]
Since \(\rho_{\boldsymbol A}(Y)=1\), the cut matrix
\(\boldsymbol A[Y,X\cup Z]\) has rank one. Moreover,
\[
\boldsymbol A[Y,X\cup Z]
=
\bigl[\boldsymbol A[Y,X]\;\;\boldsymbol 0\bigr],
\]
and \(\boldsymbol A[Y,X]\neq\boldsymbol 0\). Hence we may write
\[
\boldsymbol A[Y,X]
=
\boldsymbol a\boldsymbol b^\top,
\]
where \(\boldsymbol a\in\mathbb R^Y\) and
\(\boldsymbol b\in\mathbb R^X\) are nonzero and entrywise
nonnegative. Thus \[\boldsymbol A[Y,X\cup Z]
=
\boldsymbol a
\begin{bmatrix}
\boldsymbol b^\top&\boldsymbol 0_Z^\top
\end{bmatrix}.\]

For the cut \((Y,X\cup Z)\), the corresponding weighted sums are
\[
\sum_{i\in Y}a_i\boldsymbol v_i
\qquad\text{and}\qquad
\sum_{j\in X}b_j\boldsymbol v_j.
\]
Since \(\boldsymbol\theta\) is synchronized on both \(X\) and \(Y\),
\[
\sum_{i\in Y}a_i\boldsymbol v_i
=
a(Y)\boldsymbol v_Y,
\qquad
\sum_{j\in X}b_j\boldsymbol v_j
=
b(X)\boldsymbol v_X.
\]
Both coefficients are strictly positive, so both weighted sums
are nonzero. Corollary~\ref{lemma:1} therefore implies that they
point in the same direction. Hence $\boldsymbol v_X=\boldsymbol v_Y$,
and consequently $\theta_X\equiv\theta_Y\pmod{2\pi}$.
\end{proof}

\subsection{Proof of Theorem~\ref{main}}
\label{sec:main-proof}
Figure~\ref{fig:tree-labels} illustrates the tree labeling used in the
proof below.

\begin{figure}[H]
\centering
\begin{tikzpicture}[
    x=1cm,y=1cm,
    line cap=round,
    every node/.style={font=\small},
    vertex/.style={circle,fill=black,draw=black,inner sep=1.6pt},
    leaf label/.style={font=\small,inner sep=2pt},
    set label/.style={font=\small,align=center,inner sep=1pt},
    tree edge/.style={draw=black,line width=0.55pt}
]
  \coordinate (u) at (-1.7,0.6);
  \coordinate (v) at ( 1.7,0.6);
  \coordinate (t) at (-4.3,-0.7);
  \coordinate (w) at ( 4.3,-0.7);
  \coordinate (l1) at (-6.2,0.0);
  \coordinate (l2) at (-6.2,-1.7);
  \coordinate (l3) at (-4.3,1.5);
  \coordinate (l4) at ( 4.3,1.5);
  \coordinate (l5) at ( 6.2,0.0);
  \coordinate (l6) at ( 6.2,-1.7);

  \foreach \child/\parent in {l1/t,l2/t,t/u,l3/u,l4/v,l5/w,l6/w,w/v}
    \draw[tree edge] (\child)--(\parent);
  \draw[draw=black,line width=1.2pt] (u)--(v)
    node[midway,above=4pt,inner sep=1pt] {$e$};

  \foreach \name in {u,v,t,w,l1,l2,l3,l4,l5,l6}
    \node[vertex] at (\name) {};

  \node[set label] at (-1.7,1.65) {$u$\\[-1pt]$S_u=\{1,2,3\}$};
  \node[set label] at ( 1.7,1.65) {$v$\\[-1pt]$S_v=\{4,5,6\}$};
  \node[set label,anchor=west] at (-3.95,-1.02)
    {$t:\ S_t=\{1,2\}$};
  \node[set label,anchor=east] at ( 3.95,-1.02)
    {$w:\ S_w=\{5,6\}$};

  \node[leaf label,above left=2pt] at (l1) {$\{1\}$};
  \node[leaf label,below left=2pt] at (l2) {$\{2\}$};
  \node[leaf label,above=3pt] at (l3) {$\{3\}$};
  \node[leaf label,above=3pt] at (l4) {$\{4\}$};
  \node[leaf label,above right=2pt] at (l5) {$\{5\}$};
  \node[leaf label,below right=2pt] at (l6) {$\{6\}$};
\end{tikzpicture}
\caption{Leaf-set labels after choosing the edge $e=uv$
(thick edge). Each leaf displays its associated singleton set.
The sets $S_u$ and $S_v$ partition $V$, and each internal
label is the union of its two child labels; for example,
$S_u=S_t\cup\{3\}$.}
\label{fig:tree-labels}
\end{figure}
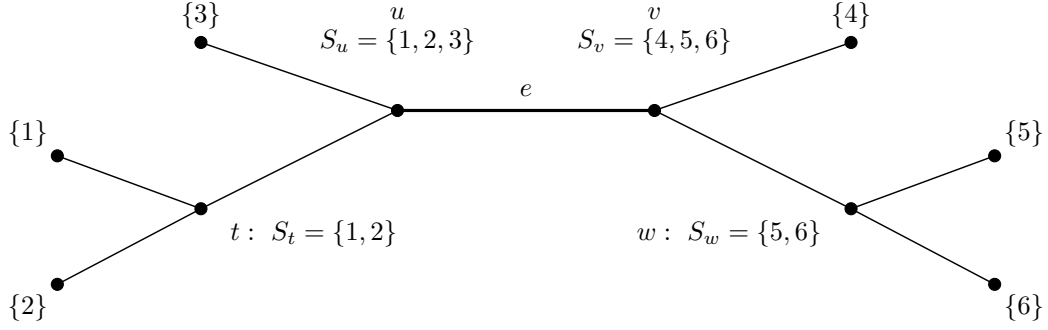

\begin{proof}
Assume $\boldsymbol\theta$ is a local minimum or a non-flat
first- and second-order point of
$E_{\boldsymbol A}(\boldsymbol\theta)$.
In either case, it satisfies \eqref{eq:foc} and \eqref{eq:soc}.
The support graph of $\boldsymbol A$ is connected by assumption.

Let $(T,L)$ be a rank-decomposition of $\boldsymbol A$ of width one.
Choose an arbitrary edge $e=uv$ of $T$.
Deleting $e$ from $T$ produces two subtrees $T_u$ and $T_v$,
which we view as rooted at $u$ and $v$, respectively.
Let $S_u,S_v\subseteq V$ be the sets corresponding, via $L$,
to the leaves of these rooted trees. Thus
$V=S_u\mathbin{\dot\cup}S_v$.

We now label each node $t$ of $T_u$ by a set $S_t\subseteq V$.
For $t\neq u$, let $e_t$ be the unique edge incident to $t$
on the path from $t$ to $u$.
Deleting $e_t$ from $T_u$ produces a component containing $t$;
view this component as a tree $T_t$ rooted at $t$.
Let $S_t$ be the set of vertices corresponding, via $L$,
to the leaves of $T_t$.
In particular, if $t=L(i)$ is a leaf of $T$, then $S_t=\{i\}$.
Apply the same procedure to the nodes of $T_v$.
Here a rooted tree consisting of a single node has that node as its leaf.

The cut $(S_t,V\setminus S_t)$ is induced by $e$ when
$t\in\{u,v\}$, and by $e_t$ otherwise.
Since $(T,L)$ has width one, $\rho_{\boldsymbol A}(S_t)\leq1$.
Connectedness of the support graph and
$\varnothing\neq S_t\subsetneq V$ imply that the corresponding
cut matrix is nonzero. Hence $\rho_{\boldsymbol A}(S_t)=1$
for every node $t$ of $T$.

We first show that $\boldsymbol\theta$ is synchronized on $S_t$
for every node $t$, by induction on $|S_t|$.
If $|S_t|=1$, there is nothing to prove.
Assume $|S_t|\geq2$. Then $t$ is not a leaf; let $t_1,t_2$
be its two children. By construction,
$S_t=S_{t_1}\mathbin{\dot\cup}S_{t_2}$, so
$S_{t_1},S_{t_2},V\setminus S_t$ partition $V$ into three
nonempty sets.
By the induction hypothesis, $\boldsymbol\theta$ is synchronized
on $S_{t_1}$ and on $S_{t_2}$.
Since
\[
\rho_{\boldsymbol A}(S_{t_1})
=\rho_{\boldsymbol A}(S_{t_2})
=\rho_{\boldsymbol A}(V\setminus S_t)
=\rho_{\boldsymbol A}(S_t)=1,
\]
Lemma~\ref{base-weighted} implies that $\boldsymbol\theta$
is synchronized on $S_{t_1}\cup S_{t_2}=S_t$.
This completes the induction.

It remains to show that the phases on $S_u$ and $S_v$ agree.
Let $\boldsymbol w_u,\boldsymbol w_v$ be their respective common unit vectors.
The cut matrix $\boldsymbol A[S_u,S_v]$ is nonzero, entrywise nonnegative,
and of rank one, so write
\[
\boldsymbol A[S_u,S_v]=\boldsymbol p\boldsymbol q^\top
\]
with nonzero, entrywise nonnegative vectors
$\boldsymbol p\in\mathbb R^{S_u}$ and $\boldsymbol q\in\mathbb R^{S_v}$.
The corresponding weighted sums are
\[
\sum_{i\in S_u}p_i\boldsymbol v_i=p(S_u)\boldsymbol w_u,
\qquad
\sum_{j\in S_v}q_j\boldsymbol v_j=q(S_v)\boldsymbol w_v.
\]
Both are nonzero because $p(S_u),q(S_v)>0$.
By Corollary~\ref{lemma:1}, they point in the same direction.
Since $\boldsymbol w_u$ and $\boldsymbol w_v$ are unit vectors,
they are equal. As $S_u\cup S_v=V$, it follows that
\[
\theta_1\equiv\cdots\equiv\theta_n\pmod{2\pi}.
\]
\end{proof}

\section{Sharpness and examples}
\label{sec:examples}
\subsection{Sharpness in the unweighted case}
\label{sec:sharpness}

Neither conclusion of Corollary~\ref{cor:distance-hereditary} extends
to all unweighted graphs of real rank-width two, as witnessed by every
cycle $C_n$ with $n\geq5$.
We use the standard twisted-state construction on rings;
see \cite{WileyStrogatzGirvan2006} for background.

\begin{proposition}
\label{prop:cycles}
Let $n\geq5$ and let $\boldsymbol A\in\{0,1\}^{n\times n}$ be the adjacency matrix of the cycle $C_n$.
Then
\[
\operatorname{rw}_{\mathbb R}(\boldsymbol A)=2.
\]
Moreover, there exists a nonsynchronized configuration $\boldsymbol\theta$
that is both a local minimum and a non-flat first- and second-order point
of $E_{\boldsymbol A}(\boldsymbol\theta)$.
The corresponding equilibrium of \eqref{kuramoto} is locally
asymptotically stable modulo global rotation.
\end{proposition}

Figure~\ref{fig:c5-twisted-state} illustrates the case $n=5$.

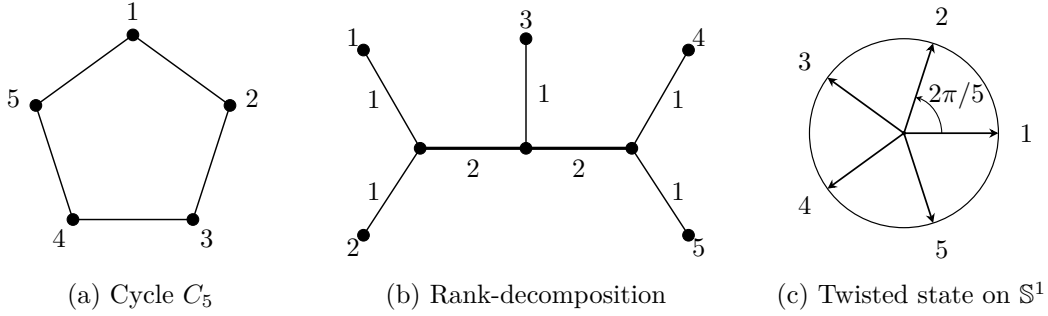
\begin{figure}[H]
\centering
\begingroup
\captionsetup[subfigure]{font=small,labelfont=normalfont,justification=centering,skip=3pt}
\tikzset{
  cycle canvas/.style={x=1cm,y=1cm,line cap=round,line join=round,
    every node/.style={font=\small,inner sep=1pt}},
  cycle vertex/.style={circle,fill=black,draw=black,inner sep=1.55pt},
  cycle edge/.style={draw=black,line width=0.55pt},
  cycle cut edge/.style={draw=black,line width=1.2pt},
  cycle leaf label/.style={inner sep=0pt,outer sep=0pt},
  phase vector/.style={draw=black,line width=0.65pt,
    -{Stealth[length=3.5pt,width=3.5pt]}}
}
\begin{minipage}{15.0cm}
\centering
\begin{subfigure}[b]{0.32\linewidth}
  \centering
  \begin{tikzpicture}[cycle canvas]
    \path[use as bounding box] (0,0) rectangle (4.8,3.75);
    \begin{scope}[shift={(2.3,1.8)}]
      \foreach \i/\angle in {1/90,2/18,3/-54,4/-126,5/162}
        \coordinate (cv\i) at (\angle:1.35);
      \draw[cycle edge]
        (cv1)--(cv2)--(cv3)--(cv4)--(cv5)--cycle;
      \foreach \i/\angle in {1/90,2/18,3/-54,4/-126,5/162} {
        \node[cycle vertex] at (cv\i) {};
        \node at (\angle:1.66) {$\i$};
      }
    \end{scope}
  \end{tikzpicture}
  \caption{Cycle $C_5$}
\end{subfigure}\hfill%
\begin{subfigure}[b]{0.32\linewidth}
  \centering
  \begin{tikzpicture}[cycle canvas]
    \path[use as bounding box] (-0.1,0) rectangle (4.7,3.75);
    \coordinate (cu) at (0.9,1.65);
    \coordinate (cv) at (2.3,1.65);
    \coordinate (cw) at (3.7,1.65);
    \coordinate (cl1) at (0.15,2.95);
    \coordinate (cl2) at (0.15,0.5);
    \coordinate (cl3) at (2.3,3.1);
    \coordinate (cl4) at (4.45,2.95);
    \coordinate (cl5) at (4.45,0.5);
    \draw[cycle edge] (cl1)--(cu)
      node[midway,left=3pt] {$1$};
    \draw[cycle edge] (cu)--(cl2)
      node[midway,left=3pt] {$1$};
    \draw[cycle edge] (cv)--(cl3)
      node[midway,right=3pt] {$1$};
    \draw[cycle edge] (cl4)--(cw)
      node[midway,right=3pt] {$1$};
    \draw[cycle edge] (cw)--(cl5)
      node[midway,right=3pt] {$1$};
    \draw[cycle cut edge] (cu)--(cv)
      node[midway,below=3pt] {$2$};
    \draw[cycle cut edge] (cv)--(cw)
      node[midway,below=3pt] {$2$};
    \foreach \name in {cu,cv,cw,cl1,cl2,cl3,cl4,cl5}
      \node[cycle vertex] at (\name) {};
    \node[cycle leaf label,above left=2pt] at (cl1) {$1$};
    \node[cycle leaf label,below left=2pt] at (cl2) {$2$};
    \node[cycle leaf label,above=4pt] at (cl3) {$3$};
    \node[cycle leaf label,above right=2pt] at (cl4) {$4$};
    \node[cycle leaf label,below right=2pt] at (cl5) {$5$};
  \end{tikzpicture}
  \caption{Rank-decomposition}
\end{subfigure}\hfill%
\begin{subfigure}[b]{0.32\linewidth}
  \centering
  \begin{tikzpicture}[cycle canvas]
    \path[use as bounding box] (0,0) rectangle (4.8,3.75);
    \begin{scope}[shift={(2.3,1.85)}]
      \draw[draw=black,line width=0.45pt] (0,0) circle[radius=1.25];
      \foreach \i/\angle in {1/0,2/72,3/144,4/216,5/288}
        \draw[phase vector] (0,0)--(\angle:1.25);
      \fill[black] (0,0) circle[radius=0.8pt];
      \foreach \i/\angle in {1/0,2/72,3/144,4/216,5/288}
        \node at (\angle:1.62) {$\i$};
      \draw[draw=black,line width=0.4pt,
        -{Stealth[length=2.5pt,width=2.5pt]}]
        (0:0.5) arc[start angle=0,end angle=72,radius=0.5];
      \node at (36:0.86) {$2\pi/5$};
    \end{scope}
  \end{tikzpicture}
  \caption{Twisted state on $\mathbb S^1$}
\end{subfigure}
\end{minipage}
\endgroup
\caption{The $1$-twisted state on $C_5$.
(a) The cycle with unit edge weights.
(b) A rank-decomposition with each tree edge labeled by its cut-rank:
two for nonpendant edges and one for pendant edges.
(c) The phase vectors $\boldsymbol v_k$ at angles $2\pi(k-1)/5$,
equally spaced on the unit circle.
This nonsynchronized configuration is both a local minimum and a non-flat
first- and second-order point. The corresponding equilibrium is locally
asymptotically stable modulo global rotation.}
\label{fig:c5-twisted-state}
\end{figure}

\begin{proof}
The rank-width assertion follows from Example~\ref{ex:cycle-rank-width}.
Label the vertices of $C_n$ cyclically by $1,\ldots,n$ and consider the
$1$-twisted state
\[
\theta_k=\frac{2\pi(k-1)}{n},
\qquad k=1,\ldots,n.
\]
Its phases are not all equal, so $\boldsymbol\theta$ is nonsynchronized.
At every vertex, the phase differences with its two neighbors are
$\pm 2\pi/n$ modulo $2\pi$. The two sine terms in \eqref{eq:foc}
therefore cancel, so $\boldsymbol\theta$ is a first-order
point of $E_{\boldsymbol A}(\boldsymbol\theta)$.

At this state, every edge satisfies $\cos(\theta_i-\theta_j)=\cos(2\pi/n)$. Therefore,
the Hessian is
\[
\boldsymbol H
=
\cos\left(\frac{2\pi}{n}\right)L_{C_n},
\]
where $L_{C_n}=2I-\boldsymbol A$ is the graph Laplacian of $C_n$
and $I$ is the $n\times n$ identity matrix. Since
$n\geq5$, we have $\cos(2\pi/n)>0$. Since $C_n$ is connected,
\[
\boldsymbol H\succeq0,
\qquad
\ker\boldsymbol H
=
\operatorname{span}\{\boldsymbol 1\},
\]
so $\boldsymbol\theta$ satisfies both the first- and second-order
conditions. Lemma~\ref{lem:flat-hessian} implies that it is non-flat. Therefore, it is a non-flat nonsynchronized first- and second-order point.

To verify local minimality, let $\boldsymbol h\in\mathbb R^n$
be an arbitrary sufficiently small perturbation and set
$g(t)=E_{\boldsymbol A}(\boldsymbol\theta+t\boldsymbol h)$
for $t\in[0,1]$.
Stationarity gives $g'(0)=0$, and
\[
g''(t)=\sum_{i=1}^n
\cos\bigl(\theta_i-\theta_{i+1}+t(h_i-h_{i+1})\bigr)
(h_i-h_{i+1})^2,
\]
where indices are interpreted modulo $n$.
Since every edge satisfies
$\cos(\theta_i-\theta_{i+1})=\cos(2\pi/n)>0$,
all cosine factors remain positive for every $t\in[0,1]$
when $\boldsymbol h$ is sufficiently small.
Consequently,
\[
E_{\boldsymbol A}(\boldsymbol\theta+\boldsymbol h)
-E_{\boldsymbol A}(\boldsymbol\theta)
=\int_0^1(1-t)g''(t)\,dt\geq0.
\]
Equality holds if and only if $h_1=\cdots=h_n$,
that is, the perturbation is a common rotation.
Thus $\boldsymbol\theta$ is a strict local minimum modulo
global rotation, and in particular a local minimum.
\end{proof}

\subsection{Sharpness in the weighted case}
\label{subsec:weighted-example}
The weighted graph in Figure~\ref{fig:intro-weights}(C) establishes
sharpness of both conclusions of Theorem~\ref{main}, even when the
support graph is complete.

\begin{example}[A weighted complete graph with a nonsynchronized local minimum]
\label{ex:weighted-k5}
Let $V=\{0,1,2,3,4\}$, with indices interpreted modulo $5$, and define
\[
\boldsymbol A_\varepsilon=
\begin{pmatrix}
0&1&\varepsilon&\varepsilon&1\\
1&0&1&\varepsilon&\varepsilon\\
\varepsilon&1&0&1&\varepsilon\\
\varepsilon&\varepsilon&1&0&1\\
1&\varepsilon&\varepsilon&1&0
\end{pmatrix}.
\]
For
\[
0<\varepsilon<\frac{7-3\sqrt5}{2},
\]
the configuration $\theta_k=2\pi k/5$ is a nonsynchronized strict
local minimum of $E_{\boldsymbol A_\varepsilon}(\boldsymbol\theta)$ modulo global
rotation and a non-flat first- and second-order point.
All off-diagonal entries of $\boldsymbol A_\varepsilon$
are strictly positive, so its support graph is $K_5$.
Moreover, $\operatorname{rw}_{\mathbb R}(\boldsymbol A_\varepsilon)=2$.
\end{example}

\begin{proof}
Every cut of a five-vertex graph has a side of size at most two, so
$\operatorname{rw}_{\mathbb R}(\boldsymbol A_\varepsilon)\leq2$.
Every rank-decomposition has a cut isolating two vertices; by cyclic
symmetry, a corresponding $2\times2$ minor has determinant
$\varepsilon(\varepsilon-1)$ if the vertices are adjacent on the cycle,
and $1-\varepsilon$ otherwise. Both are nonzero in the stated range,
so the rank-width is two.

At each vertex, the phase differences occur in opposite pairs
$\pm2\pi/5$ with weight $1$ and $\pm4\pi/5$ with weight
$\varepsilon$. The sine terms therefore cancel, proving stationarity.

Set $c=\cos(2\pi/5)=(\sqrt5-1)/4$ and
$d=\cos(4\pi/5)=-(\sqrt5+1)/4$. By \eqref{hessian},
\[
(\boldsymbol H\boldsymbol x)_k
=c(2x_k-x_{k-1}-x_{k+1})
+\varepsilon d(2x_k-x_{k-2}-x_{k+2}).
\]
The Fourier vectors $x_k=\exp(2\pi\mathrm{i}mk/5)$,
$m=0,\ldots,4$, where $\mathrm{i}^2=-1$, diagonalize this circulant
matrix, with eigenvalues
\[
\lambda_m
=2c\bigl(1-\cos(2\pi m/5)\bigr)
+2\varepsilon d\bigl(1-\cos(4\pi m/5)\bigr).
\]
In particular,
\[
\begin{aligned}
\lambda_0&=0,\\
\lambda_1=\lambda_4
&=\frac{3\sqrt5-5-(5+3\sqrt5)\varepsilon}{4},\\
\lambda_2=\lambda_3&=\frac{\sqrt5}{2}(1-\varepsilon).
\end{aligned}
\]
The assumed range of $\varepsilon$ makes all four nonzero-mode
eigenvalues strictly positive. Thus $\boldsymbol H$ is positive
definite on $\boldsymbol 1^\perp$, proving strict local minimality modulo
global rotation. Moreover, $\boldsymbol H\succeq0$ and
$\ker\boldsymbol H=\operatorname{span}\{\boldsymbol 1\}$, so
Lemma~\ref{lem:flat-hessian} gives non-flatness.
Since the phases are not all equal, the configuration
is not a global minimum by \eqref{eq:global-energy-gap}.
\end{proof}

For example, $\varepsilon=1/10$ gives nonzero Hessian eigenvalues
approximately $0.13435$ and $1.00623$, each with multiplicity two.
The corresponding equilibrium is locally asymptotically stable modulo
global rotation, because the linearization of \eqref{kuramoto} is
$-\boldsymbol H$. Hence it attracts an open set of initial conditions
and also rules out almost-everywhere convergence to a synchronization configuration.
In particular, the distance-hereditary support condition cannot replace
the weighted cut-rank condition in Theorem~\ref{main}.

\section*{Acknowledgements}
The author thanks Monique Laurent for helpful discussions, careful
reading of the manuscript, and valuable suggestions on the proofs
and presentation.

\section*{Declaration of AI use}
Large language models were used to improve the presentation of the results through language editing and assistance with TikZ code for figures. All mathematical results and proofs are the author's own, and the author takes full responsibility for all content.

\bibliographystyle{postive-alpha}
\bibliography{postive}
\end{document}